\documentclass[11pt,reqno,tbtags]{amsart}

\usepackage{amsthm,amssymb,amsfonts,amsmath}
\usepackage{amscd} 
\usepackage{tikz-cd}
\usepackage{tikz}
\usepackage{mathrsfs} 
\usepackage{subcaption}
\usepackage{graphicx}
\usepackage{vmargin}
\usepackage[colorlinks=true]{hyperref}
\usepackage{xspace}
\usepackage{array}
\usepackage{xcolor}
\usepackage{stmaryrd} 
\usepackage{mathtools}
\usepackage{wrapfig}
\usepackage[shortlabels]{enumitem}
\usepackage{scalerel}

\usepackage[numbers]{natbib}

\providecommand{\Z}{}
\providecommand{\Q}{}
\renewcommand{\Z}{\mathbb{Z}}
\renewcommand{\Q}{\mathbb{Q}}

\newcommand{\rE}{\mathrm{E}} 

\newcommand{\rV}{\mathrm{V}} 
\newcommand{\rp}{\mathrm{p}}
\renewcommand{\rm}{\mathrm{m}}

\newcommand{\rt}{\mathrm{t}}
\newcommand{\rc}{\mathrm{c}}

\DeclareRobustCommand*{\ora}{\overrightarrow}

\newcommand{\ux}{\underline{x}}
\newcommand{\uy}{\underline{y}}

\newtheorem{theo}{Theorem}
\newtheorem*{theo*}{Theorem}
\newtheorem{lemm}[theo]{Lemma}
\newtheorem{prop}[theo]{Proposition}
\newtheorem{property}[theo]{Property}
\newtheorem{coro}[theo]{Corollary}

\newtheorem{rema}[theo]{Remark}
\newtheorem{clai}[theo]{Claim}
\numberwithin{equation}{section}
\numberwithin{theo}{section}

\theoremstyle{definition}
\newtheorem{defi}[theo]{Definition}

\newcommand{\cO}{\mathcal{O}} 
 
\newcommand{\out}{\mathrm{deg_{out}}}
\newcommand{\ind}{\mathrm{deg_{in}}}

\newcommand{\rh}{\hslash}
\newcommand{\halfsquare}{
  \tikz[scale=0.6, baseline=-0.5ex]{
    \draw[line width=0.28mm, black] (0,0) rectangle (0.3,0.3);
    \fill[line width=0.28mm, black] (0,0) -- (0.3,0) -- (0.3,0.3) -- cycle;}}
\newcommand{\sq}{{\scaleto{\halfsquare}{4pt}}}

\newcommand{\halfsquareempty}{
  \tikz[scale=0.6, baseline=-0.5ex]{
    \draw[line width=0.28mm, black] (-0.1,-0.1) rectangle (0.2,0.2);}}
\newcommand{\sqempt}{{\scaleto{\halfsquareempty}{4pt}}}

\newcommand{\dtrum}{$\mathsf{d}$-trumpet}
\newcommand{\dcorn}{$\mathsf{d}$-cornet}

\graphicspath{ {./images/}}
\begin{document}

\title[Blossoming bijection for bipartite maps and Ising model]{Blossoming bijection for bipartite maps:\\ a new approach via orientations\\ and applications to the Ising model} 

\author{Marie Albenque}

\author{Laurent Ménard}

\author{Nicolas Tokka}

\begin{abstract} 
We develop a new bijective framework for the enumeration of bipartite planar maps with control on the degree distribution of black and white vertices. Our approach builds on the blossoming-tree paradigm, introducing a family of orientations on bipartite maps that extends Eulerian and quasi-Eulerian orientations and connects the bijection of Bousquet-Mélou and Schaeffer to the general scheme of Albenque and Poulalhon. This enables us to generalize the Bousquet-Mélou and Schaeffer's bijection to several families of bipartite maps. 

As an application, we also derive a rational and Lagrangian parametrization with positive integer coefficients for the generating series of quartic maps equipped with an Ising model, which is key to the probabilistic study of these maps. 
\end{abstract} 

\maketitle

\tableofcontents

\section*{Introduction}

Planar maps are central objects in enumerative combinatorics and mathematical physics, where they serve as discrete models for random surfaces. The enumeration of planar maps dates back to the work of Tutte in the 1960s, who derived closed-form enumerative formulas for numerous families of maps in a series of seminal papers~\cite{Tutte62,Tutte63, Tutte64, Tutte68, Tutte95}. His most relevant work in the context of the present article is the enumeration of planar maps with control on vertex degrees when all the degrees are even~\cite{Tutte62}, further generalized to any vertex degrees by Bender and Canfield \cite{BeCa94}. 

Tutte's approach consists in encoding recursive decomposition of maps into functional equations for their generating series. The functional equations that arise are generally complicated, and one usually needs to introduce extra parameters -- the so-called catalytic variables -- in order to write them down. This approach was later extended and placed into a systematic framework, and equations with catalytic variables continue to play a central role in modern enumerative combinatorics. This is for example illustrated in the papers by Bousquet-Mélou and Jehanne~\cite{BousquetMelouJehanne} and Bernardi and Bousquet-Mélou~\cite{BernardiBousquetMelou}. We also refer to the survey by Bousquet-Mélou~\cite{BousquetMelou06}.

Independently, in the physics literature, the problem of enumerating planar maps emerged in connection with matrix models, see for example the work of 't Hooft \cite{THooft74}, and was solved by Brezin, Itzykson, Parisi and Zuber in 1978 \cite{BrezinItzyksonParisiZuber}. We refer the reader to the books by Lando and Zvonkine~\cite{LandoZvonkin} and Eynard~\cite{Eynard16}  for more information on this point of view. 

While giving an answer to the enumeration of maps, both approaches fail to explain the simplicity of the formulas obtained, which are closely related to the enumeration of trees. It turns out that this connection can be directly explained by explicit bijections between maps and some families of decorated trees. The first explicit bijection was constructed by Cori and Vauquelin~\cite{CoriVauquelin}, and was subsequently revitalized in Schaeffer’s PhD thesis~\cite{SchaefferPhd}. This revival paved the way for a broad range of bijective methods for studying and enumerating various families of maps. For example, the works of Schaeffer \cite{Scha97} and of Bouttier, Di Francesco and Guitter \cite{BouttierDiFrancescoGuitter_Census, BouttierDiFrancescoGuitter_Mobiles} which provide bijective proofs of the enumerative formulas for maps with control on vertex degrees aforementioned.

These bijections not only yield efficient enumeration formulas, but also provide structural insight, and were instrumental in the study of large random planar maps, for instance in the works of 
Chassaing and Schaeffer~\cite{ChassaingSchaeffer04}, Chassaing and  Durrhus~\cite{ChassaingDurhuus06}, Miermont~\cite{Miermont_Brownian} and Le Gall~\cite{LeGall_Uniqueness}, see also the survey~\cite{LeGallMiermont12}.

\smallskip

In this article, we continue the bijective study of planar maps, and deal with the enumeration of \emph{bipartite planar maps}, which are maps where vertices can be properly bicolored. Equivalently, bipartite maps are planar maps where all the faces have even degree. More precisely, we address the problem of enumerating bijectively bipartite planar maps, while controlling the distribution of vertex degrees of black and white vertices independently. This problem generalizes the enumeration of classical maps with control on vertex degrees. Indeed, each map can be transformed into a bipartite map, by inserting a black vertex of degree 2 in the middle of each edge. 

The enumeration of bipartite planar maps originated in the physics literature, where they appeared in connection with the so-called ``2-matrix model'' studied by Itzykson and Zuber \cite{Itzykson1980}. The motivation behind their study stems from their connection with the Ising model on planar maps, solved by Boulatov and Kazakov  via matrix integrals \cite{KazakovBoulatov87,Kazakov86}. The first bijective enumeration of bipartite maps was obtained by Bousquet-Mélou and Schaeffer in~\cite{BousquetMelouSchaeffer_Bipartite}. To this end, they built an explicit bijective correspondence between bipartite maps and a family of decorated bicolored trees. This allowed them to rederive fully rigorously the results of Kazakov~\cite{Kazakov86}. Shortly after, another bijective proof of this result, relying on a different family of decorated trees -- called mobiles -- was developed by Bouttier, Di Francesco and Guitter in~\cite{BouttierDiFrancescoGuitter_Mobiles}.

More recently, unified bijective frameworks have been developed to provide generic methods for constructing bijections between maps and tree-like families. Notable examples include the schemes of Bernardi and Fusy~\cite{BernardiFusy_Boundaries, BernardiFusy_Girth, BernardiFusy_Hypermaps, BernardiFusy_TrigQuadPent}, as well as the scheme of Albenque and Poulalhon~\cite{AlbenquePoulalhon_Generic}. Both approaches extend a construction introduced by Bernardi~\cite{Bernardi07} for maps equipped with a canonical orientation of their edges.

\smallskip

The starting point of the present work is to introduce a new family of orientations on bipartite maps that allows to interpret the Bousquet-Mélou and Schaeffer's bijection as an instance of the bijective scheme of Albenque-Poulalhon. This family of orientations generalizes the classical Eulerian and quasi-Eulerian orientations. 

To describe more precisely the main contribution of this paper, some terminology is needed (precise definitions will be given later in Section~\ref{sec2}). A \emph{blossoming tree} is a tree where vertices can carry opening and closing stems, which are half-edges oriented either from or towards their incident vertex. The \emph{charge} of a blossoming tree is defined as the difference between the number of  its closing and opening stems. The \emph{closure} of a blossoming tree is the planar maps obtained by matching opening and closing stems cyclically around the tree, see Figure~\ref{fig:intro}. In~\cite{BousquetMelouSchaeffer_Bipartite}, a family of trees with some charge constraints was defined and was shown to be in bijection with bipartite planar maps through the closure operation. In this paper, we generalize this result in two directions.

First, we relax the \emph{balancedness assumption} of~\cite{BousquetMelouSchaeffer_Bipartite}, in which only trees such that their root lies in the outer-face of their closure map are considered. We instead obtain a bijection between not necessarily balanced blossoming trees and planar bipartite maps with a marked outer face, see Theorem~\ref{theo: Bijection Bipartite maps}. The main consequence is of enumerative nature, since removing the balancedness assumption simplifies considerably the enumeration of the family of trees obtained.  

Second, we give an interpretation of the closure of trees with \emph{non-zero charge} in Theorem~\ref{theo: Bijection well-rooted trees with general charge}. In that case some stems remain unmatched, and by connecting them to an additional vertex we obtain a rooted bipartite map with an additional marked vertex and some connectivity constraints, see Figure~\ref{fig:intro}. This interpretation of trees with non-zero charge gives a combinatorial explanation of the rational parametrization for hypermaps given in \cite[Chapter 8]{Eynard16}. Recently, another combinatorial interpretation of this parametrization was given by Albenque and Bouttier~\cite{AlbenqueBouttier_slices}, using the theory of slices. This allows them to give a combinatorial proof of several enumerative formulas for hypermaps, using as fundamental building blocks some particular subfamilies of hypermaps on the cylinder, called \emph{trumpets} and \emph{cornets}. These families correspond exactly (by duality) to the families we obtain as the closure of trees, which confirms the important role of trumpets and cornets in the context of the topological recursion, see e.g. Eynard~\cite{Eynard16}.

Another unified bijective treatment for hypermaps has been developed by Bernardi and Fusy~\cite{BernardiFusy_Hypermaps}, building on their previous works for general maps~\cite{BernardiFusy_TrigQuadPent, BernardiFusy_Girth, BernardiFusy_Boundaries}. The main difference with our approach is that they develop a setting specially tailored to treat the case of hypermaps, whereas we use a scheme developed for general maps and apply it to a family of orientations specific to bipartite maps. Their bijective scheme allows them to generalize existing bijections, with a full control on girth and cycle lengths constraints in the hypermaps. In particular, they rederive Bousquet-Mélou and Schaeffer's blossoming bijection and Bouttier, Di Francesco and Guitter's mobile bijection as special cases of their construction. Interestingly, we can also recover the latter bijection by applying the original bijective framework for general maps introduced in~\cite{BernardiFusy_TrigQuadPent} to our orientations.

\smallskip

Finally, let us mention an important byproduct of our bijection for maps decorated with an Ising model. As already mentioned, it is classical that the generating series of bipartite maps and of maps decorated with an Ising model are connected through a change of variables, see e.g.~\cite{BousquetMelouSchaeffer_Bipartite} and Section~\ref{sec: Connection with bipartite maps}. In particular, this connection allowed Bousquet Mélou and Schaeffer to derive a rational parametrization for the generating series of quartic maps with an Ising model in~\cite{BousquetMelouSchaeffer_Bipartite}.

Our new bijection allows us to derive an equivalent rational and Lagrangian parametrization of this generating function that has a combinatorial interpretation, see Theorem~\ref{theo: Lagragian parametrization Ising GF}. A consequence of this combinatorial interpretation is that the parameter series has non-negative integer coefficients. Both this property and the Lagrangian form of the parametrization allow for a detailed analysis of the generating series of quartic maps with an Ising model. In turn, this allows to study probabilistic properties of these maps as was done for example in~\cite{ChenTurunen22, ChenTurunen20, IsingAMS, AlbenqueMenard, Turunen}. Tokka uses the parametrization of Theorem~\ref{theo: Lagragian parametrization Ising GF} to study random quartic maps with an Ising model in presence of an external magnetic field in \cite{T25}.

\begin{figure}[t!]
	\centering
	\includegraphics[width=0.32\linewidth ,page=1]{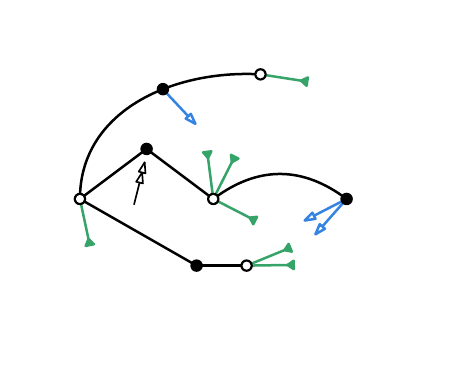}\hspace{-1.2cm}
	\includegraphics[width=0.32\linewidth ,page=2]{ClosureIntro.pdf}\quad
	\includegraphics[width=0.32\linewidth ,page=3]{ClosureIntro.pdf}
	\caption{\label{fig:intro} A blossoming bipartite tree with charge 4 (left), its complete closure with a additional black vertex of degree 4 (middle), and the  corresponding  dual hypermap with a additional marked face of degree 4 (right).}
\end{figure}

Let us also mention that our bijection for bipartite maps through orientation could be an important tool to address the problem of enumerating bipartite planar maps with prescribed vertex degrees in higher genus, building upon the work of Lepoutre~\cite{Lepoutre19}. This could pave the way for exploring the Ising model on planar maps in higher genus using a bijective approach. Note that the recent work by Bouquet-Mélou, Carrance and Louf~\cite{BousquetMelouCarranceLouf_Ising}, studies cubic maps of arbitrary genus equipped with an Ising model with different methods. They provide inequalities for the coefficients of the generating function using equations of the KP hierarchy. \\

\textbf{Outline.} Section \ref{sec2} reviews the background on planar maps, orientations and recalls the bijective scheme between maps and the so-called blossoming trees, which will be central in this work. 
Section  \ref{sec3} introduces our new family of orientations on bipartite maps and establishes the two main bijections of this paper (see Theorems~\ref{theo: Bijection Bipartite maps} and~\ref{theo: Bijection well-rooted trees with general charge}), which recover and generalize the bijection by Bousquet-Mélou and Schaeffer for bipartite planar maps.

Section \ref{sec4} explores enumerative consequences of our bijections for bipartite maps with a marked face, dual maps of the so-called trumpets and cornets, and doubly rooted bipartite planar maps.
Section~\ref{sec5} deals with maps decorated with an Ising model in general and Section~\ref{sec6} deals specifically with quartic maps equipped with an Ising model.

Appendix~\ref{AppA} explains why our orientations encode geodesic properties of the maps and justifies how they generalise Eulerian and quasi-Eulerian orientations. Appendix~\ref{AppB} applies the unified bijective scheme of Bernardi-Fusy to our family of orientation and shows that it allows us to recover the so-called mobile bijection for maps of Bouttier, Di Francesco and Guitter. \\

\textbf{Acknowledgments.} 
M.A. and N.T.’s research were supported by the ANR grant IsOMa
(ANR-21-CE48-0007).

\newpage
\section{Preliminaries and background}\label{sec2}

	\subsection{Planar maps and families of maps} 
		
		\subsubsection{Planar maps and plane maps}
A \emph{planar map} is a proper embedding of a connected planar graph on the 2-dimensional sphere $\mathbb{S}^2$, considered up to orientation-preserving homeomorphisms. 

\begin{figure}
\centering
\includegraphics[scale=0.8, page=3]{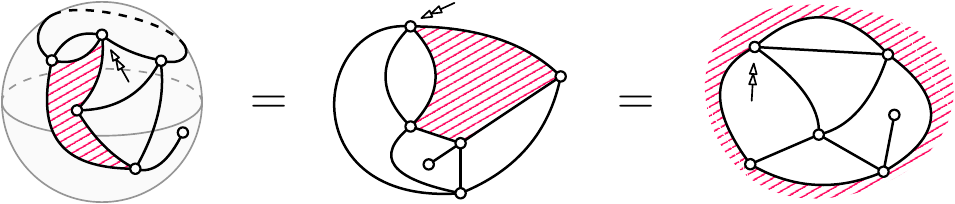}
\caption{\label{fig:planeMap} Three representations of the same plane map: as a map embedded in the sphere with a marked dashed face (left), and as a map embedded in the plane with the rooted corner in the outer face (middle), or with the marked face as the outer face (right). In the rest of this article, we will always use the last representation.}
\end{figure}

\emph{Edges} and \emph{vertices} of a map are the natural counterparts of edges and
vertices of the underlying graph. The \emph{faces} of a map $\rm$ are the
connected components of the complement of the embedded graph. The sets of its vertices, edges and faces are respectively denoted by $\mathrm{V}(\rm)$, $\mathrm{E}(\rm)$ and $\mathrm{F}(\rm)$. Note that loops and multiple edges are allowed.


Each edge $e$ of $\rm$ can be split at its middle point into two \emph{half-edges} $h_1$ and $h_2$. Thus, we often write $e=\{h_1,h_2\}$. 
The set of half-edges of $\rm$ is denoted by $\mathrm{H}(\rm)$. For any vertex $v\in \mathrm{V}(\rm)$, and any half-edge $h\in \mathrm{H}(\rm)$, we write $h\sim v$, if $h$ is incident to $v$.

For a directed edge $\overrightarrow{e}$, we write $\overrightarrow{e}=(h_1,h_2)$, where $h_1$ and $h_2$ are the two half-edges corresponding to $e$, and such that $h_1$ (resp. $h_2$) is incident to the tail of $e$ (resp. to the head of $e$). Finally, if there is no ambiguity (a.k.a. no loops nor multiple edges; for example, in the case of trees), write $\overrightarrow{e}\coloneqq\overrightarrow{uv}$ with $u,v\in \mathrm{V}(\rm)$, and set $h_{\overrightarrow{uv}}\coloneqq h_1$ and $h_{\overrightarrow{vu}}\coloneqq h_2$.

The planar embedding of a map fixes the cyclical order of half-edges around each vertex, which
defines readily a \emph{corner} as a couple of consecutive half-edges
around a vertex. The set of corners of $\rm$ is denoted by $\mathrm{C}(\rm)$.  There is a bijective correspondence between $\mathrm{H}(\rm)$ and $\mathrm{C}(\rm)$ by mapping each half-edge with the corner that follows it in counterclockwise order around its vertex. 
The \emph{degree}
of a vertex or a face is defined as the number of its incident corners. In
other words, it counts incident edges, with multiplicity 2 for each
loop (in the case of vertex degree) or for each bridge (in the case of
face degree).

\smallskip

To avoid dealing with symmetries, planar maps will always be \emph{rooted}, meaning that one of their corners -- called the \emph{root corner} -- is distinguished (and indicated by a double arrow on figures). The vertex and the face incident to this corner are called the \emph{root vertex} and the \emph{root face}, respectively. We will typically denote the root vertex of a map by $\rho$.

A \emph{pointed map} is a map with an additional marked vertex. \medskip

A \emph{plane map} is a map with an additional marked face. The term ``plane'' comes from the fact that, by taking this marked face as the \emph{outer infinite face}, a plane map admits a canonical embedding in the plane via stereographic projection, see Figure~\ref{fig:planeMap}. The marked face is called the \emph{outer face}, the corners and the vertices incident to the outer face are called the \emph{outer corners} and \emph{outer vertices}, respectively. This notion will be central in this article, since the geometric constructions we consider are often easier to describe and analyse in the plane rather than in the sphere. 

We emphasize that the outer infinite face of a plane map is not necessarily its root face, contrary to what is often done in the literature. 

\medskip

A \emph{plane tree} is a plane map with only one face. A \emph{planted tree} is a plane tree which is rooted at a dangling half-edge. In other words, a planted tree is obtained from a rooted  plane tree by inserting a half-edge at its root corner and designating this new half-edge as the root.

\subsubsection{Bipartite maps, Eulerian maps and duality}\label{sub:duality}

\begin{figure}
	\centering
	\includegraphics[width=0.32\linewidth ,page=1]{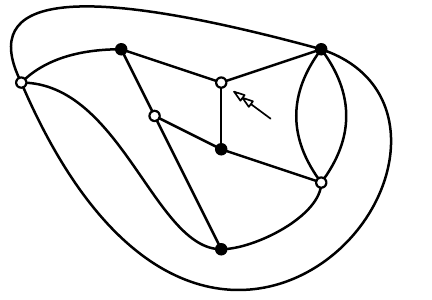}
	\includegraphics[width=0.32\linewidth ,page=2]{BipartiteEulerian.pdf}
	\includegraphics[width=0.32\linewidth ,page=3]{BipartiteEulerian.pdf}
	\caption{\label{fig:ex_Eulerian_bipartite}A bipartite map with a proper 2-coloring of its vertices (left), its dual with the correponding coloring of its faces (middle), and with the canonical direction of its edges and its directed geodesic labeling (right). The bipartite map verifies $\Delta_\circ=4$ and $\Delta_\bullet=5$.}
\end{figure}
A map $\rm$ is called \emph{bipartite} if the set of its vertices can be partitioned into two disjoint subsets, $\rV_\bullet$ and $\rV_\circ$, such that every edge connects a vertex of $\rV_\bullet$ to a vertex of $\rV_\circ$. In other words, $\rm$ admits a proper 2-coloring of its vertices in black (corresponding to $\rV_\bullet$) and white (corresponding to $\rV_\circ$). It is easy to see that a planar map is bipartite if and only if all its faces have even degree.

Observe that there are only two proper 2-colorings, which differ one from another by switching the color of all vertices. 
Throughout this paper, \emph{we will always assume that a bipartite map is endowed with one of these two colorings}.
We denote by $\Delta_\circ(\rm)$ and $\Delta_\bullet(\rm)$ the maximal degree of its white and black vertices, respectively. When the context is clear, we simply write $\Delta_\circ$ and $\Delta_\bullet$, see Figure~\ref{fig:ex_Eulerian_bipartite}. For any $e\in \rE(\rm)$, we denote respectively by $h_\circ(e)$ and $h_\bullet(e)$ the white and black endpoints of $e$.

We denote by $\mathcal{M}$ the set of rooted bipartite planar maps, and for any $d\geq 0$, we denote by $\mathcal{M}^{(d)}$ its subset composed of the maps with maximal vertex degree $d$. Similarly, we denote by $\bar{\mathcal{M}}$ the set of rooted bipartite \emph{plane} maps, and for any $d\geq 0$, we denote by $\bar{\mathcal{M}}^{(d)}$ its subset restricted to the maps with maximal vertex degree $d$ . 
\smallskip

The dual map $\rm^\dagger$ of the planar map $\rm$ is defined as follows. Vertices of $\rm^\dagger$ correspond to faces of $\rm$ and, for each edge $e$ of $\rm$, there is an edge $e^\dagger$ in $\rm^\dagger$ that links the two vertices of $\rm^\dagger$ corresponding to the faces of $\rm$ incident to $e$. See Figure~\ref{fig:ex_Eulerian_bipartite} for an illustration.

A map is called~\emph{bicolorable} if its faces can be bicolored in black and white, in such a way that every edge separates a black and a white face. It is easy to see that a planar map is bicolorable if and only if all its vertices have even degree; such a map is often referred to as a \emph{Eulerian map}, or a \emph{hypermap}. These maps are dual to bipartite maps.

Similarly to bipartite maps, an Eulerian map admits only two bicolorings of its faces, which differ one from another by flipping all the colors. Again, throughout this paper, \emph{we will always assume that an Eulerian map is endowed with one of its two proper colorings}, see Figure~\ref{fig:ex_Eulerian_bipartite}.

The edges of an Eulerian map are \emph{canonically oriented}, by requiring that white faces lie on the left of the directed edges (and, black faces on their right). Equivalently, the contour of each white face is directed in the counterclockwise direction (and, of each black face in the clockwise direction).

\subsection{Orientations, fractional orientations and \texorpdfstring{$\alpha$}{alpha}-orientations}
		
The definitions of orientations presented in this section follow \cite{Felsner04,BernardiFusy_Girth}. More general definitions were introduced in \cite{BernardiFusy_Girth,BernardiFusy_Boundaries,BernardiFusy_TrigQuadPent}, but this level of generality is not necessary to capture the combinatorics of the models studied in this paper.
				
		\subsubsection{Orientations}\label{ssub:def_orientations}
\begin{figure}
	\centering
	\includegraphics[width=0.35\linewidth ,page=1]{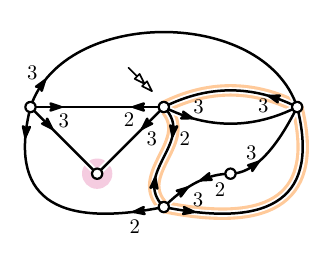}\qquad
	\includegraphics[width=0.35\linewidth ,page=2]{Orientations.pdf}
	\caption{\label{fig: Examples of orientations} A map with a non-minimal (see the orange-highlighted cycle) and non-accessible (see the pink-highlighted vertex) 3-fractional orientation (left), and a bipartite map with its minimal $\alpha_4$-orientation (right).}
\end{figure}

An \emph{orientation} of a planar map $\rm$ is a mapping $\cO$ from $\mathrm{H}(\rm)$ to $\mathbb{Z}_{\geq 0}$. The orientation is viewed and will be represented as the value of an outgoing flow through the corresponding half-edge, see Figure~\ref{fig: Examples of orientations}. 

For any oriented map $(\rm,\cO)$, a directed edge $\overrightarrow{e}=(h_1,h_2)$ of $\rm$ is said to be \emph{forward} if $\cO(h_1)>0$, and to be \emph{saturated} if ${\cO(h_2)=0}$. 
When the tail and the head of a saturated edge $\overrightarrow{e}$ are known, say $u$ and $v$, respectively, we say that \emph{$e$ is saturated from $u$ to $v$}.

Next, for $u,v\in \mathrm{V}(\rm)$, a directed path $\rp\coloneqq(\overrightarrow{e_1},\ldots,\overrightarrow{e_p})$ from $u$ to $v$ in $\rm$ is said to be \emph{forward} if for any $i\in\{0,\ldots,p-1\}$, the directed edge $\overrightarrow{e_i}$ is forward. The vertex $v$ is said to be \emph{accessible} from $u$, if there exists a forward path from $u$ to $v$. Moreover, an orientation $\cO$ is said to be \emph{root-accessible} (or simply \emph{accessible}), if its root vertex is accessible from every vertex of $\rm$, see Figure~\ref{fig: Examples of orientations}. 

In a plane map $\rm$ endowed with an orientation, a \emph{clockwise} (respectively \emph{counterclockwise}) cycle is a forward cycle such that the marked face of $\rm$ lies on its left (respectively on its right). We call \emph{minimal} any orientation without counterclockwise cycles.

		\subsubsection{Fractional orientations and \texorpdfstring{$\alpha$}{alpha}-orientations}
Fix an integer $k>0$ and $\rm$ a planar map. An orientation of $\rm$ is said to be \emph{$k$-fractional} if for every edge $e=\{h_1,h_2\}\in \mathrm{E}(\rm)$, we have $\cO(h_1)+\cO(h_2)=k$, see Figure~\ref{fig: Examples of orientations}. From now on, all the orientations considered are \emph{fractional}, meaning that they are $k$-fractional for some $k\in \mathbb{Z}_{>0}$. 

Let $\cO$ be an orientation of $\rm$. Then, for $v\in \mathrm{V}(\rm)$, the outdegree and the indegree of $v$ for $\cO$ -- respectively denoted by $\out(v)$ and $\ind(v)$ -- are defined by: 

\begin{equation*}
	\out(v)\coloneqq\sum_{h\sim v}\cO(h)\qquad \text{and} \qquad \ind(v)\coloneqq\sum_{\substack{h_1\sim v\\ \{h_1,h_2\}\in \mathrm{E}(\rm)}}\cO(h_2).
\end{equation*}
Note that if $\rm$ is endowed with a $k$-fractional orientation, we have: 
\[
\out(v)+\ind(v)=k\cdot\deg(v) \quad \text{ for any }v\in V(\rm).
\]

Fix a function $\alpha:\mathrm{V}(\rm)\rightarrow \mathbb{Z}_{\geq 0}$. An \emph{$\alpha$-orientation} $\cO$ of $\rm$ is a fractional orientation such that for every vertex  $v\in \mathrm{V}(\rm)$, we have: $\out(v)=\alpha(v)$. If $\rm$ admits an $\alpha$-orientation, then the function $\alpha$ said to be \emph{feasible} on $\rm$. Note that ``classical'' $\alpha$-orientations introduced by Felsner in~\cite{Felsner04} correspond to taking $k=1$.

The following proposition will be central in this work:
\begin{prop}\cite{Felsner04,BernardiFusy_Girth}\label{prop: Unique minimal orientation}
If $\alpha:\mathrm{V}(\rm)\rightarrow \mathbb{Z}_{\geq 0}$ is feasible on a plane map $\rm$, then there exists a unique minimal $\alpha$-orientation on $\rm$.

Moreover, when $\rm$ is rooted, if one $\alpha$-orientation of $\rm$ is accessible, then all other $\alpha$-orientations are accessible. In particular, the minimal orientation is also accessible.
\end{prop}

	\subsection{Blossoming maps and bijective scheme}
		\subsubsection{Blossoming maps and blossoming trees}

A \emph{blossoming map} is a plane map in which each \emph{outer corner} can carry some half-edges. These half-edges are called \emph{closing stems} (for ingoing half-edges) and \emph{opening stems} (for outgoing half-edges), also commonly known in the literature as buds and leaves (see \cite{BousquetMelouSchaeffer_Bipartite,BouttierDiFrancescoGuitter_Census}). A \emph{blossoming tree} is a blossoming map with only one face. A \emph{planted blossoming tree} is a blossoming tree rooted at a dangling half-edge, which is neither an opening nor a closing stem.

An orientation $\cO$ of a blossoming map is defined as an orientation for maps (see Section~\ref{ssub:def_orientations}), subject to the additional condition that $\cO(h)=0$, for any closing stem $h$. Moreover, if $\cO$ is $k$-fractional, we also require that $\cO(h)=k$, for any opening stem $h$.

\medskip

The \emph{charge} of a blossoming map is defined as the difference between the number of its closing stems and the number of its opening stems. Furthermore, for a blossoming tree $\rt$, we define the \emph{charge} of any $v\in \rV(\rt)$ -- denoted by $\rc_\rt(v)$, or $\rc(v)$ if the context is clear -- as the charge of the subtree of $\rt$ rooted at $v$. 
		
\begin{figure}[t!]
	\centering
	\includegraphics[width=1\linewidth ,page=2]{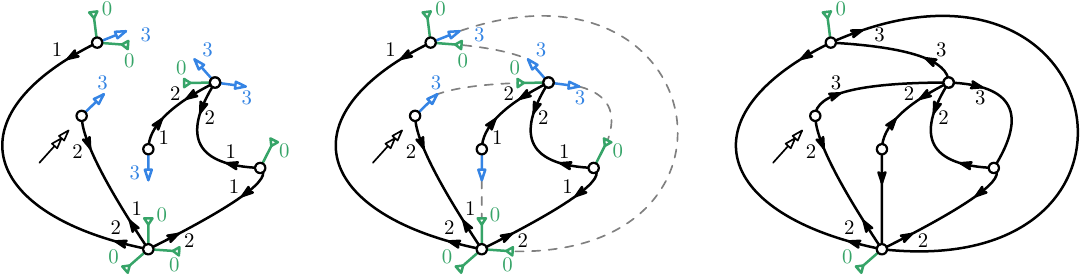}
	\caption{\label{fig: Closure blossoming tree} A blossoming tree of charge $2$ with a $3$-fractional orientation (left). The same tree with the matching of stems in dashed lines (middle). Its closure, which is a blossoming map with two unmatched closing stems (right).}
\end{figure}		
		
		\subsubsection{Closure of blossoming maps}
Following~\cite{Scha97,BouttierDiFrancescoGuitter_Census,PouSch06,Bernardi07,AlbenquePoulalhon_Generic}, given a blossoming map $\rm$, we define its \emph{closure} as follows. Consider the cyclic sequence of opening and closing stems, obtained by turning clockwise around the border of its marked face. This induces a partial matching between opening and closing stems, as in a parenthesis word. Note, that some opening or closing stems remain unmatched if the charge of the whole tree is negative or positive, respectively. For any matched pair made of an opening and a closing stem, merge them together so as to create a new edge, in such a way that the outer infinite face lies on its left (when oriented from the opening stem to the closing stem), see Figure~\ref{fig: Closure blossoming tree}. It is easy to prove that the resulting plane blossoming map does not depend on the order on which each local closure is performed. Also note that the vertex degree distribution is preserved through the closure operation. 

The closure of a blossoming map with charge 0 is a plane map (i.e with no remaining unmatched stems). In contrast, the closure of a blossoming map with charge $k>0$ (resp. $k<0$) yields a blossoming map with $k$ unmatched closing (resp. opening) stems.

If a blossoming map $\rm$ is equipped with an orientation $\cO$, then its closure naturally inherits this orientation -- still denoted by $\cO$, by a slight abuse of notation. Since $\cO(h)=0$ for any closing stem, the edges created during the closure are all saturated.

Note that both accessibility and minimality are preserved under the closure operation. Accessibility follows immediately, while minimality holds because the closure process is performed in the clockwise direction and so does not create any counterclockwise cycles.

		\subsubsection{Bijective scheme}
		The inverse of the closure operation is \emph{a priori} not well-defined. Indeed, a map can be obtained as the closure of any of its spanning trees. 
The following result by Albenque and Poulalhon, which generalizes earlier results of Poulalhon and Schaeffer~\cite{PouSch06} and Bernardi~\cite{Bernardi07}, gives a framework in which a canonical inverse can be defined:

\begin{theo}[\cite{AlbenquePoulalhon_Generic}, Corollary 2.4]\label{theo: AlbenquePoulalhonTheorem - Blossoming bijection for maps with an minimal accessible orientation}
Let $\rm$ be a rooted plane map with a minimal accessible orientation $\cO$. Then, there exists a unique blossoming tree $\rt_\rm$ of charge $0$, equipped with an accessible orientation, whose closure yields $(\rm,\cO)$.

Moreover, the statement holds as well if $\rm$ is a blossoming map with $k$ unmatched closing stems (respectively, opening stems). In that case, $\rt_\rm$ is a blossoming tree of charge $k$ (respectively, $-k$).
\end{theo}

Let us emphasize two crucial points about this correspondence. Consider a rooted plane map $\rm$ and its associated blossoming tree $\rt_\rm$. First, the edges of $\rm$ that are not in $\rt_\rm$ are necessarily saturated, since only edges of this type are created during the closure procedure. 
Second, $\rt_\rm$ and $\rm$ share the same vertex degree distribution, as the closure procedure preserves vertex degrees.

\section{New family of orientations and blossoming bijections for bipartite maps}\label{sec3}
In this section, we define a new family of feasible $\alpha$-orientations on bipartite planar maps. Then, we apply the bijective scheme from Theorem~\ref{theo: AlbenquePoulalhonTheorem - Blossoming bijection for maps with an minimal accessible orientation} to both recover and generalize the bijection of bipartite maps defined by Bousquet-Mélou and Schaeffer in~\cite{BousquetMelouSchaeffer_Bipartite}.

	\subsection{Definitions and first properties of \texorpdfstring{$\alpha_d$}{alpha_d}-orientations} 
In all this section, $\rm\in\mathcal{M}$ is a rooted bipartite  planar map (endowed with one of its two proper colorings). For $d\geq 1$, let $\alpha_d : \rV(\rm)\rightarrow \mathbb{Z}_{\geq 0}$  be the function defined by: 
\begin{equation}\label{eq: Condition to be an alpha d orientation}
    \alpha_d(v) \coloneqq
    \begin{cases*}
      d\,\deg(v) & if $v$ is black, \\
      \deg(v)  & if $v$ is white.
    \end{cases*}
\end{equation}
The function $\alpha_{d}$ is feasible on $\rm$, as a $(d+1)$-fractional orientation. Indeed, it suffices to set:
\begin{equation*}
   \tilde\cO(h)\coloneqq 
   \begin{cases}
     d & \text{ for any }h\text{ incident to a black vertex},\\
     1 & \text{ for any }h\text{ incident to a white vertex}.
   \end{cases} 
 \end{equation*} 
In the orientation $\tilde\cO$, each edge is forward in both directions, so that $\tilde\cO$ is clearly  accessible. Thus, as a direct consequence of Proposition~\ref{prop: Unique minimal orientation}, we have:

\begin{figure}[t]
	\centering
	\includegraphics[width=0.25\linewidth
	,page=1]{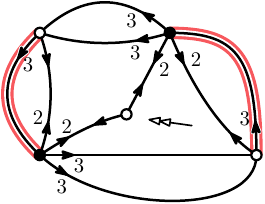}\qquad\qquad
	\includegraphics[width=0.25\linewidth
	,page=2]{AlphaOrientation.pdf}\qquad\qquad
	\includegraphics[width=0.25\linewidth
	,page=3]{AlphaOrientation.pdf}
	\caption{\label{fig: Examples of alpha orientations} A rooted bipartite plane map with $\Delta_\circ=4$, equipped with its minimal $\alpha_d$-orientation, shown from left to right for $d=2$, $d=4$, and $d=6$.\\
	Observe that for $d=2$, two edges (highlighted in red) are saturated from their white endpoints to their black endpoints, whereas for $d\geq 4$, all saturated edges are oriented from their black endpoints to their white endpoints. We can observe the stability property from Property~\ref{claim: Unique minimal accessible alpha d orientation for a bipartite map} between $d=4$ and $d=6$.}
\end{figure}

\begin{property}\label{claim: Unique minimal accessible alpha d orientation for a bipartite map}
For any $d\geq 1$, every bipartite plane map $\rm$ admits a unique minimal accessible $\alpha_{d}$-orientation.
\end{property}

\begin{rema}
Note that every bipartite planar map $\rm$ endowed with an $\alpha_{d}$-orientation is in fact \emph{strongly connected}, meaning that every vertex is accessible from every other vertex.
\end{rema}

We now state key properties of $\alpha_d$-orientations. 
\begin{property}\label{claim: Saturated edges are from black to white endpoints}
Fix $d\geq \Delta_\circ(\rm)$. Then, in any $\alpha_d$-orientation $\cO$, every saturated edge is oriented from its black endpoint to its white endpoint.
\end{property}
\begin{proof}
By definition, we have $\cO(h) \leq \Delta_\circ(\rm) < d+1$, for any half-edge $h\in \mathrm{H}(\rm)$ adjacent to a white vertex. The claim follows since $\cO$ is $(d+1)$-fractional.  
\end{proof}

Note that this result is false for general values of $d$, see Figure~\ref{fig: Examples of alpha orientations}. As a consequence of the previous claim, we have the following stability property for minimal $\alpha_d$-orientations, illustrated on Figure~\ref{fig: Examples of alpha orientations}:
\begin{property}\label{claim:stability for orientations}
Fix $d_1\geq d_2\geq\Delta_\circ(\mathrm{m})$ and write respectively $\cO_1$ and $\cO_2$ for the minimal $\alpha_{d_1}$-orientation and $\alpha_{d_2}$-orientation on $\rm$. Then:  
\begin{equation}\label{eq:defO1}
    \cO_1(h) =
    \begin{cases*}
      \cO_2(h)   & if $h$ is adjacent to a white vertex, \\
      \cO_2(h) + (d_1-d_2)  & if $h$ is adjacent to a black vertex.
    \end{cases*}
\end{equation}
\end{property}
\begin{proof}
Let $\cO_2$ be the minimal $\alpha_{d_2}$-orientation of $\rm$, and define $\cO_1$ as in \eqref{eq:defO1}. Then, $\cO_1$ is clearly an $\alpha_{d_1}$-orientation. To prove that $\cO_1$ is minimal, we only need to prove that if an edge $e=(h_\bullet,h_\circ)$ is saturated in $\cO_2$, then it remains saturated in the same direction in $\cO_1$. By Property~\ref{claim: Saturated edges are from black to white endpoints}, if $e$ is saturated, it is from $h_\bullet$ to $h_\circ$, and hence $\cO_2(h_\circ)=0$. By definition of $\cO_1$, we also have that $\cO_1(h_\circ)=0$, which concludes the proof. 
\end{proof}
An alternative proof of this property based on the geodesic properties of $\alpha_d$-orientations will be given in Appendix~\ref{AppA}, see Remark~\ref{rem:proofclaim}.

	\subsection{Recovering the BMS bijection: the case of plane maps}
		\subsubsection{Blossoming bijection and well-charged trees}
In this section, we apply Theorem~\ref{theo: AlbenquePoulalhonTheorem - Blossoming bijection for maps with an minimal accessible orientation} to bipartite maps endowed with their minimal $\alpha_d$-orientation, to obtain a new bijection between bipartite maps and the family of \emph{well-charged trees} introduced in \cite{BousquetMelouSchaeffer_Bipartite}, which we now define. 

\begin{figure}[t!]
\centering
\includegraphics[scale=0.9, page=]{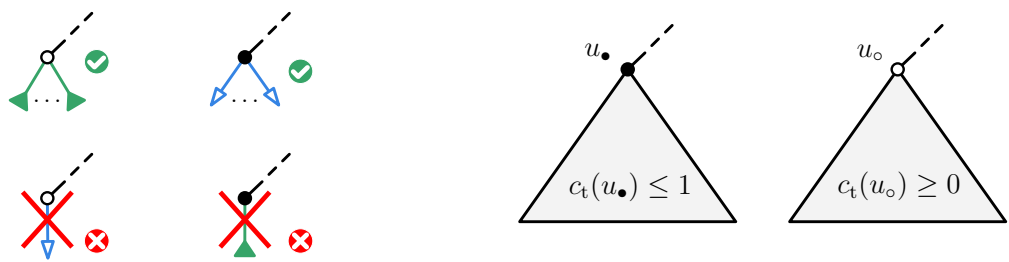}
\caption{\label{fig:well-charged trees}Constraints on well-charged trees.}
\end{figure}

Let us recall that the charge of a blossoming tree is the difference between its number of closing stems and of opening stems, then a \emph{well-charged (blossoming) tree} is a rooted bipartite blossoming tree satisfying the following conditions, illustrated in Figure~\ref{fig:well-charged trees}: 
\begin{itemize}
\item black vertices have no incident closing stems, and every non-root black vertex has charge at most one,
\item white vertices have no incident opening stems, and every non-root white vertex has non-negative charge.
\end{itemize}

A well-charged tree is said to be \emph{black} or \emph{white}, depending on the color of its root. For any $k\in\mathbb{Z}$ and $d\geq 0$, we denote by $\mathcal{T}_k$ the set of well-charged trees of total charge $k$, and by $\mathcal{T}_k^{(d)}$ the subset consisting of trees whose maximal vertex degree is $d$. 


Recall the definition of $\bar{\mathcal{M}}$ in Section~\ref{sub:duality}, the main result of this section is the following bijection, illustrated in Figure~\ref{fig: closure well-charged trees}:

\begin{theo}\label{theo: Bijection Bipartite maps}
The closure operation is a one-to-one correspondence between $\mathcal{T}_0$ and $\bar{\mathcal{M}}$.
Moreover, this correspondence preserves the vertex-degree distribution and the color of the root vertex.
\end{theo}

\begin{rema}
This result is closely related to the bijection of Bousquet-Mélou and Schaeffer~\cite{BousquetMelouSchaeffer_Bipartite} for rooted bipartite planar maps. Before giving the proof, we take a moment to highlight the key differences between the two constructions.

A way to specialize Theorem~\ref{theo: Bijection Bipartite maps} to the case of planar maps instead of plane maps, is to restrict our  attention to the case where the marked outer face of the map  coincides with its root face. On the tree side, this corresponds to considering the \emph{balanced} well-charged trees introduced in~\cite{BousquetMelouSchaeffer_Bipartite}. These are precisely the well-charged trees whose closure has the property that the outer face coincides with the root face.

	The proof of the bijection is simpler in the case of planar maps, and is based on a recursive decomposition. But it comes at the price that the enumeration of balanced well-charged trees is much more complicated than the enumeration of well-charged trees, see~\cite[Section~5]{BousquetMelouSchaeffer_Bipartite}. In particular, the enumeration relies on the additional technical assumption that the map is rooted at a black vertex of degree 2. This assumption can then be dropped by a re-rooting procedure, at the cost of an additional integration step. In the case of maps with valences $2$ and $4$, the integral is evaluated explicitly in~\cite[Section~7.3]{BousquetMelouSchaeffer_Bipartite}, and an explicit parametrization of their weighted generating function is established.
	
	In our case, the enumeration of non-necessarily balanced well-charged trees is much simpler. To recover the enumeration of planar maps, an integration step is needed to remove the marking of a face, this is done in the case of valences $2$ and $4$ in Section~\ref{sec: Bib Quartic}.
	\end{rema}

The rest of this section is devoted to the proof Theorem~\ref{theo: Bijection Bipartite maps}, which is divided into two parts. First, we define a family of blossoming trees -- referred to as $\alpha_d$-trees, and formally defined below -- which arise from applying the general bijective framework of Theorem~\ref{theo: AlbenquePoulalhonTheorem - Blossoming bijection for maps with an minimal accessible orientation} to the family of bipartite maps equipped with their canonical $\alpha_d$-orientations. In the second part, we prove that $\alpha_d$-trees are in bijection with well-charged trees. This is summarized in the following diagram:\smallskip

\begin{center}
    \begin{tikzcd}
\overline{\mathcal{M}} \arrow[rrrr, leftrightarrow, "\text{Theorem}~\ref{theo: Bijection Bipartite maps}"] \arrow[d,equal] 
&  &
&  & \mathcal{T}_0 \arrow[d, equal] \\
\bigcup\limits_{d\geq 0}\overline{\mathcal{M}}^{(d)} \arrow[rr, "\text{Prop.}~\ref{prop:Bij_Bipartite_Alpha}", leftrightarrow]    
&  &  \bigcup\limits_{d\geq 0}{\tiny \left\{\begin{array}{c} \alpha_d\text{-trees}\\ 
\textrm{with charge }0 \end{array}\right\}} \arrow[rr,"\text{Prop.}~\ref{prop:bij_charge0}", leftrightarrow] 
&  & \bigcup\limits_{d\geq 0}\mathcal{T}_0^{(d)} 
    \end{tikzcd}
\end{center}

		\subsubsection{Definition and closure of $\alpha_d$-trees }\label{sub: alpha_d-trees}

We start with the definition of $\alpha_d$-trees.

\begin{defi}
Fix $d\geq 1$. An \emph{$\alpha_d$-tree} is a blossoming tree $\rt$ such that: 
\begin{itemize}
    \item $\rt$ is \emph{bipartite},
    \item $\rt$ has degree at most $d$, i.e. $\Delta(\rt)\leq d$,
    \item $\rt$ is equipped with an accessible $\alpha_d$-orientation.
\end{itemize}
Depending on the color of its root vertex, an $\alpha_d$-tree is said to be \emph{black} or \emph{white}.
\end{defi}

Note that if a tree admits an $\alpha_d$-orientation, it is unique and can be constructed recursively, starting from the leaves. Moreover, the existence of such an orientation imposes some constraints on the shape of the tree. In particular, we have: 

\begin{lemm}\label{lemm: Opening and Closing stems constraint on alpha_d-trees}
Fix $d\geq 1$. In an $\alpha_d$-tree, all opening stems are incident to black vertices, and all closing stems to white vertices. 

Consequently, the closure of any $\alpha_d$-tree is a bipartite map. 
\end{lemm}

\begin{figure}[t!]
  	\centering
  	\includegraphics[width=0.31\linewidth ,page=1]{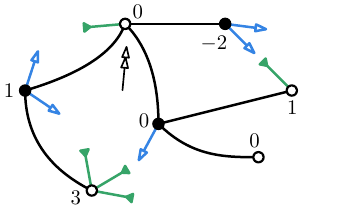}
  	\includegraphics[width=0.31\linewidth ,page=2]{CompleteWellChargedTrees.pdf}
  	\includegraphics[width=0.31\linewidth ,page=3]{CompleteWellChargedTrees.pdf}
 	\caption{\label{fig: closure well-charged trees} A well-charged tree of charge $0$ with its charge function (left), and with its $\alpha_4$-orientation (middle), recall that this means that opening stems count for $5$ and closing stems count for $0$. Its closure, which is endowed with its unique minimal accessible $\alpha_4$-orientation (right).}
\end{figure}

\begin{proof}
Let $(\rt,\cO)$ be an $\alpha_d$-tree, where $\cO$ stands for its $\alpha_d$-orientation. Fix $v\in \rV_\circ(\rt)$. By definition, we have  $\out(v)=\deg(v)\leq \Delta(\rt) \leq d$. Since $\cO(h)=d+1$ for any opening stem $h$, it implies that no opening stem can be incident to $v$. 

We now consider $v\in \rV_\bullet(\rt)$. Write respectively $k, n_{op}, n_{cl}\geq 0$ for the number of its adjacent vertices, its incident opening stems and its incident closing stems, so that $k+n_{cl}+n_{op}=\deg(v)$. By definition of an $\alpha_d$-orientation, we have: 
\begin{equation*}
d\, \left(k+n_{cl}+n_{op}\right) = \alpha_d(v) =\sum_{h\sim v}{\cO(h)} = \left(d+1\right)n_{op} + \sum_{w\sim v}{\cO(h_{\overrightarrow{vw}})}.
\end{equation*}
We give an upper bound for the last sum. Let ${w\in\rV(\rt)}$ be a child of $v$. Since $\cO$ is accessible, we have $\cO(h_{\overrightarrow{wv}})\geq 1$, so that $\cO(h_{\overrightarrow{vw}})\leq d$. Note that if $v$ is the root of $\rt$, it has $k$ children, otherwise, it has $k-1$ children. We can then bound the right-hand side of the last equation, to get: 
\begin{equation*}
d\left(k+n_{cl}+n_{op}\right) \leq \begin{cases}
\left(d+1\right)n_{op} + dk,& \text{if }v\text{ is the root vertex,}\\
\left(d+1\right)n_{op} + dk +1, &\text{otherwise.}
\end{cases}
\end{equation*}
In both cases, this yields $d n_{cl} \leq d-n_{cl}$, using the fact that ${n_{op}= \deg(v)-n_{cl}-k}$ and that $\deg(v)\leq d$. Since $k\geq 1$ if $v$ is not the root vertex, this implies that $n_{cl}=0$, and concludes the proof.
\end{proof}

The first step in proving Theorem~\ref{theo: Bijection Bipartite maps} is the following proposition, illustrated in~Figure~\ref{fig: closure well-charged trees}:

\begin{prop}\label{prop:Bij_Bipartite_Alpha}
For $d\geq 1$, the closure operation is a one-to-one correspondence between the set of $\alpha_d$-trees of charge $0$, and $\bar{\mathcal{M}}^{(d)}$.
Moreover, this correspondence preserves the vertex-degree distribution and the color of the root vertex.
\end{prop}

\begin{proof}
By Lemma~\ref{lemm: Opening and Closing stems constraint on alpha_d-trees}, the closure of any $\alpha_d$-tree is bipartite. Hence, the closure operation defines a mapping from the set of $\alpha_d$-trees to $\bar{\mathcal{M}}^{(d)}$. Theorem~\ref{theo: AlbenquePoulalhonTheorem - Blossoming bijection for maps with an minimal accessible orientation} ensures that this closure is a bijection, by equipping each map of $\bar{\mathcal{M}}^{(d)}$ with its canonical $\alpha_d$-orientation, whose existence is granted by Property~\ref{claim: Unique minimal accessible alpha d orientation for a bipartite map}.
\end{proof}

\begin{rema}\label{rem:infinite_bijection}
The previous proposition associates to any bipartite $\rm$ an infinite number of $\alpha_d$-trees. Namely, one for each value of $d\geq \Delta(\rm)$. However all these trees have the same shape. Indeed, by Property~\ref{claim:stability for orientations}, the saturated edges are the same in any minimal $\alpha_d$-orientation on $\rm$ whenever $d\geq \Delta(\rm)$. Moreover, by the same proposition, their orientations only differ by an additive shift for half-edges incident to black edges. 
\end{rema}

		\subsubsection{Equivalence between $\alpha_d$-trees and  well-charged trees}\label{sec: Planted trees}
Given Proposition~\ref{prop:Bij_Bipartite_Alpha}, to prove Theorem~\ref{theo: Bijection Bipartite maps}, it suffices to establish the following correspondence between $\alpha_d$-trees and well-charged trees: 
\begin{prop}\label{prop:bij_charge0}
Fix $d\geq 1$, up to forgetting the orientation, the set of $\alpha_d$-trees of charge 0 corresponds to $\mathcal{T}_0^{(d)}$. 
\end{prop}

We will prove this result by induction. A \emph{planted $\alpha_d$-tree} $\rt$ is a planted  blossoming tree, endowed with an $\alpha_d$-orientation $\cO$ with the following constraint. Write $\rh$ for the root half-edge of $\rt$,  and define the \emph{excess} of $\rt$ by $\mathrm{exc}(\rt)\coloneqq\cO(\rh)$. Then, we require:
\begin{itemize}
 	\item $1\leq \mathrm{exc}(\rt) \leq d+1$, if the root vertex is black,
 	\item $1\leq \mathrm{exc}(\rt) \leq d$, if the root vertex is white.
 \end{itemize}
\begin{figure}[t]
  	\centering
  	\includegraphics[width=0.9\linewidth ,page=1]{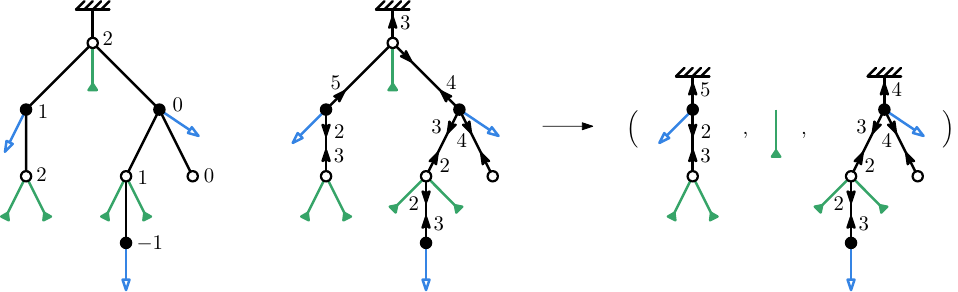}
 	\caption{\label{fig: decomposition planted tree} A well-charged white planted tree of charge $2$ represented with its charge function (left), and with its minimal accessible $\alpha_4$-orientation (middle). Its decomposition into two black well-charged planted trees and a closing stem (right). Both black trees are also endowed with their unique minimal accessible $\alpha_4$-orientation.}
\end{figure}

Planted $\alpha_d$-trees appear naturally in the recursive decomposition of $\alpha_d$-trees, see Figure~\ref{fig: decomposition planted tree}. Indeed, a direct adaptation of the proof of Lemma~\ref{lemm: Opening and Closing stems constraint on alpha_d-trees} ensures that:
\begin{clai}\label{claim:dec_alphad}
	Let $\rt$ be a white (resp. black) planted $\alpha_d$-tree rooted at a half-edge $\rh$. Let $h\neq \rh$ be a half-edge incident to the root vertex. Then: 
	\begin{itemize}
	    \item Either $h$ is a closing (resp. opening) stem, 
	    \item Or $h$ is part of an edge $e=\left\{h,h'\right\}$, and in that case the subtree of $\rt$ planted at $h'$ is a black (resp. white) planted $\alpha_d$-tree.
	\end{itemize}
\end{clai}
\smallskip

Similarly, a \emph{planted well-charged tree} $\rt$ is a planted tree, satisfying the same charge constraints as a well-charged tree, and rooted at a dangling half-edge (which is neither an opening nor a closing stem). 

Then, we have the following one-to-one correspondence, see Figure~\ref{fig: decomposition planted tree}:
\begin{lemm}\label{lemm: Charge-Excess correspondence}
Fix $d\geq 1$. For $1\leq i\leq d+1$, the operation that consists in forgetting the orientation of planted $\alpha_d$-trees yields a bijection between:
\begin{enumerate}
	\item[1.] The set of black planted $\alpha_d$-trees with excess $i$, and the set of black planted well-charged trees with maximal degree $d$, and of charge $i-d$.  
\end{enumerate}
And, for $1\leq i\leq d$, between:
\begin{enumerate}
 	\item[2.]  The set of planted white $\alpha_d$-trees with excess $i$, and the set of white planted well-charged trees with maximal degree $d$, and of charge $i-1$.   
\end{enumerate}
\end{lemm}

\begin{proof}
We proceed by induction on the number of vertices. First, we establish that planted $\alpha_d$-trees are planted well-charged trees. Fix $1\leq i\leq d$, and let $\rt$ be a planted white $\alpha_d$-tree with excess $i$. We decompose $\rt$ at its root vertex $\rho$. Let $k+1$ be the degree of $\rho$ and $l$ be the number of its incident edges, so that $\rho$ is incident to $k-l$ closing stems. Next, write $(\rt_1,\ldots,\rt_{l})$ for the sequence of planted subtrees corresponding to the children of $\rho$. By Claim~\ref{claim:dec_alphad}, $\rt_i$ is a black planted $\alpha_d$-tree, for any $1\leq i \leq l$.

Let $\cO$ be the orientation on $\rt$. Since $\rho$ is white and $\deg(\rho)=k+1$, we have: 
\begin{equation*}
k+1 = \sum_{h\sim \rho}{\cO(h)} =  i + \sum_{j=1}^l\big( d+1 - \mathrm{exc}(\rt_j)\big). 
\end{equation*} 
Recall that $\rc(\rt)$ denotes the charge of $\rt$. By the induction hypothesis, we have $\mathrm{exc}(\rt_j)=d+\rc(\rt_j)$, for any $1\leq j \leq l$, so that: 
\[
k+1 = i+l - \sum_{j=1}^l \rc(\rt_j).
\]
Finally, we get:
\begin{equation*}
\rc(\rt) = \sum_{j=1}^l \rc(\rt_j) + (k-l) = i-1\geq 0,
\end{equation*}
which concludes the case of white planted $\alpha_d$-trees. The case of black planted $\alpha_d$-trees is similar. 

\medskip

Reciprocally, fix $c\geq 0$ and $d\geq c+1$, and let $\rt$ be a planted white well-charged tree of charge $c$ and with vertex degree at most $d$. We define recursively an orientation $\cO$ on $\rt$. As before, let $k+1$ be the degree of $\rho$, let $l$ be the number of its incident edges and write $(\rt_1,\ldots,\rt_l)$ for the sequence of its planted subtrees. 
Moreover, we denote by $h_1,\ldots,h_l$ the half-edges that connect $\rho$ to $\rt_1,\ldots,\rt_l$, and by $h_{l+1},\ldots,h_k$ the closing stems incident to $\rho$.

Fix $j\in\{1,\ldots,l\}$. By the induction hypothesis, $\rt_j$ can be endowed with a unique orientation that turns it into a planted $\alpha_d$-tree with excess $d+\rc(\rt_j)$. It only remains to define the orientation of the half-edges incident to $\rho$. We set
\begin{equation*}
\cO(h_j) = \begin{cases}
1-\rc(t_j),& \text{for } 1\leq j\leq l\text{,}\\
0, &\text{for }l+1\leq j\leq k\text{,}
\end{cases}
\end{equation*}
and we set $\cO(\rh)=\rc(\rt)+1$ for the root half-edge $\rh$.  This is legitimate since $1-\rc(\rt_j)=1+d-\mathrm{exc}(t_j)\in\left\{0,\ldots,d\right\}$, for any $1\leq j\leq l$. Then, we have: 
\begin{equation*}
	\sum_{h\sim \rho}\cO(h) = \cO(\rh) + \sum_{j=1}^l\cO(h_j) = \rc(\rt) + 1 + l - \sum_{j=1}^l\rc(\rt_j),
\end{equation*}
and hence $\sum_{h\sim \rho}\cO(h)=\deg(\rho)$. Thus, $\cO$ is an $\alpha_d$-orientation which is clearly accessible. This concludes the case of white planted well-charged trees. The case of black planted well-charged trees is similar.

\medskip
Since there exists at most one orientation with prescribed out-degree on a tree, it follows that this operation is the reciprocal of forgetting the orientation of a planted $\alpha_d$-tree. This concludes the proof.
\end{proof}

\begin{proof}[Proof of Proposition~\ref{prop:bij_charge0}]
Let $\rt \in \mathcal{T}^{(d)}_0$. Define $\rho$, $k$, $l$ and $(\rt_1,\ldots,\rt_l)$ as in the proof of Lemma~\ref{lemm: Charge-Excess correspondence}.

For any $1\leq j\leq l$, endow $\rt_j$ with its canonical orientation given by Lemma~\ref{lemm: Charge-Excess correspondence}. Since $\rc(\rt)=0$, it follows that $\sum_j\rc(\rt_j)+(k-l)=0$, and so $\sum_{j}\left(d+1-\mathrm{exc}(\rt_j)\right)=k$. Hence, $\rt$ can be uniquely endowed with an $\alpha_d$-orientation following the same idea as in the proof of Lemma~\ref{lemm: Charge-Excess correspondence}.  

Reciprocally, we can prove along the same lines that forgetting the orientation of a white $\alpha_d$-tree of charge 0 yields a white well-charged tree of charge 0. 

The case of black-rooted trees is exactly similar and is left to the reader. 
\end{proof}

\begin{proof}[Proof of Theorem~\ref{theo: Bijection Bipartite maps}]
Fix $d\geq 1$, combining Proposition~\ref{prop:Bij_Bipartite_Alpha} and Proposition~\ref{prop:bij_charge0} gives a bijection between $\bar{\mathcal{M}}^{(d)}$ and $\mathcal{T}_0^{(d)}$, which preserves the vertex-degree distribution and the color of the root vertex.

This gives the existence of (at least) one bijection between $\bar{\mathcal{M}}$ and $\mathcal{T}$, by picking a canonical value of $d$ for any $\rm \in \bar{\mathcal{M}}$. Hence, this proves the theorem. In addition, observe that the bijection does not depend on this choice of $d$. Indeed, by Remark~\ref{rem:infinite_bijection}, for a fixed map, the family of $\alpha_d$-trees obtained by Proposition~\ref{prop:Bij_Bipartite_Alpha} all correspond to the same well-charged tree.
\end{proof}

	\subsection{Extending the BMS bijection: the case of \dtrum s~and \dcorn s}
		In this section, we consider the closure of well-charged trees with non-zero charge. We establish a new bijective correspondence between these trees and a family of bipartite maps with an additional marked vertex, which we call \dtrum s and \dcorn s. 
		
		\subsubsection{\dtrum s~and \dcorn s}\label{sec: def cuts trumpets corners}
\begin{defi}\label{def: cuts}
Consider a planar map $\rm$ with an additional marked vertex $\tau$. We write $\rho$ for the root vertex of $\rm$. Then, a \emph{cut} of $(\rm,\tau)$ is defined as a partition of $\mathrm{V}(\rm)$ into two sets $C=(R,S)$ such that $\rho\in R$ and $\tau\in S$. 
Its \emph{cut-set} is the subset of edges that connect a vertex in $R$ to a vertex in $S$, and its \emph{weight} is the number of such edges. 

Moreover, if $\rm$ is bipartite, $C=(R,S)$ is said to be \emph{black} (resp. \emph{white}) if all vertices in $S$ incident to edges of the cut-set are black (resp. white).

A black (resp. white) cut is said to be \emph{minimal} if its weight is minimal among all the black (resp. white) cuts. 
\end{defi}

\begin{figure}[t]
  	\centering
  	\includegraphics[width=0.31\linewidth ,page=2]{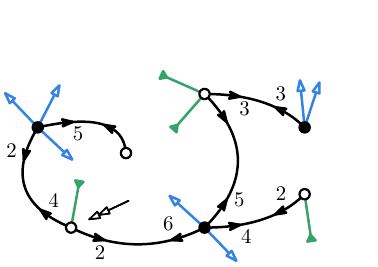}
  	\includegraphics[width=0.31\linewidth ,page=3]{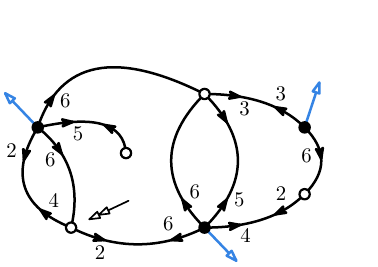}
  	\includegraphics[width=0.31\linewidth ,page=4]{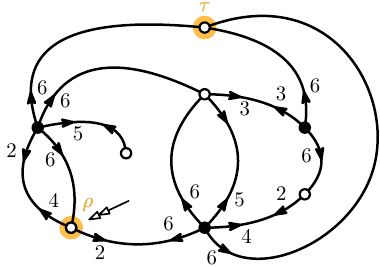}
 	\caption{\label{fig: complete closure and corner} A well-charged tree of charge $-3$ (left), its closure (middle), and its complete closure $(\rm,\rho,\tau)$ which is a \dcorn\, with $\deg(\tau)=3$ (right). Each of these is endowed with its minimal $\alpha_{5,3}^+$-orientation.} 
\end{figure}

Note that the trivial partition $(\rV\setminus\{\tau\},\{\tau\})$ defines a cut -- called the \emph{trivial cut} -- which has the same color as $\tau$, and whose weight is equal to $\deg(\tau)$. This justifies the following definition: 
\begin{defi}
A map with an additional marked vertex $(\rm,\tau)$ is said to be \emph{tight} if the trivial partition is minimal, and is said to be \emph{strictly tight} if it is the unique minimal cut of this color.

A \emph{tight} pair $(\rm,\tau)$ such that $\tau$ is black is called a \emph{\dtrum}, and a \emph{strictly tight} pair $(\rm,\tau)$ such that $\tau$ is white is called a \emph{\dcorn}.
\end{defi}

\begin{rema}
By duality, rooted planar maps with an additional marked vertex correspond to rooted planar maps with an additional marked face, colloquially known as \emph{annular maps} or \emph{maps on the cylinder}.

Annular maps satisfying a tightness condition were first considered in~\cite{BernardiFusy_TrigQuadPent} to enumerate $d$-angulations of girth $d$. By duality, bipartite maps correspond to Eulerian maps. In that setting, annular maps with tightness conditions  -- called trumpets and cornets -- were introduced and enumerated in~\cite[Section~3.3]{AlbenqueBouttier_slices}. They correspond exactly to the dual of \dcorn{s} and \dtrum{s}.
\end{rema}

\begin{rema}
Note that there is an asymmetry in the definition of \dtrum{s} and \dcorn{s}. Indeed, we do not consider the family of strictly tight triples in which $\tau$ is black, nor the family of tight triples in which $\tau$ is white. 

This asymmetry, already present in the definitions of trumpets and cornets in~\cite{AlbenqueBouttier_slices}, stems from the fact that the definition of $\alpha_d$-orientations -- and hence of $\alpha_d$-trees -- introduces an asymmetry between black and white vertices. The closure operation then produces families of objects that are asymmetric in black and white, as stated in the next theorem.
\end{rema}

\medskip

The closure $\Psi(\rt)$ of a well-charged tree $\rt$ with non-zero charge is a blossoming map with  unmatched closing stems if $c(\rt)>0$, and unmatched opening stems otherwise. The \emph{complete closure} of $\rt$ is the map obtained from $\Psi(\rt)$ as follows: add a new marked vertex $\tau$ in the outer face, and replace each unmatched stem by an edge connecting its extremity to $\tau$, see Figure~\ref{fig: complete closure and corner}. To ensure that the complete closure is bipartite, we set that $\tau$ is black if $c(\rt)>0$ and white otherwise.
\medskip

The main result of this section is the following:
\begin{theo}\label{theo: Bijection well-rooted trees with general charge}
Let $k>0$. The complete closure operation is a one-to-one correspondence between ${\mathcal{T}}_k$, and the set of \dtrum{s} $(m,\tau)$, with $\deg(\tau)=k$.

Similarly, the complete closure operation is a one-to-one correspondence between ${\mathcal{T}}_{-k}$, and the set of \dcorn{s} $(m,\tau)$, with $\deg(\tau)=k$.

Moreover, these bijections preserve the vertex degree distribution and the color of the root vertex. 
\end{theo}



The general strategy to establish Theorem~\ref{theo: Bijection well-rooted trees with general charge} is similar to the proof of Theorem~\ref{theo: Bijection Bipartite maps}. First, we define a generalization of $\alpha_d$-orientations that characterize \dtrum s and \dcorn s (Section~\ref{sub:alpha+-}). Then, we apply the general bijective scheme of Theorem~\ref{theo: AlbenquePoulalhonTheorem - Blossoming bijection for maps with an minimal accessible orientation}, and prove that the blossoming trees obtained are in bijection with well-charged trees (Section~\ref{sec:alphadktrees}).

\subsubsection{Definition and characterization of $\alpha_{d,k}^{+/-}$-orientations}\label{sub:alpha+-}
Fix $d \geq k \geq 1$, and let $(\rm,\rho,\tau)$ be a bipartite planar map $\rm$, rooted at $\rho$ and with a marked vertex $\tau$. We introduce two functions $\alpha_{d,k}^-: \mathrm{V}(\rm)\rightarrow \mathbb{Z}_{\geq 0}$ and $\alpha_{d,k}^+ : \mathrm{V}(\rm)\rightarrow \mathbb{Z}_{\geq 0}$, which extend the definition of $\alpha_d$, as follows:
\begin{align}\label{eq:alphadk}
\alpha_{d,k}^-(v) &\coloneqq \alpha_{d}(v) - k\,\mathbf{1}_{\left\{v=\rho\right\}} + k\,\mathbf{1}_{\left\{v=\tau\right\}}, \\ \alpha_{d,k}^+(v) &\coloneqq \alpha_{d}(v) + k\,\mathbf{1}_{\left\{v=\rho\right\}}- k\,\mathbf{1}_{\left\{v=\tau\right\}}.
\end{align}

In the next lemma, we prove that $\alpha_{d,k}^{+/-}$-orientations characterize \dtrum s and \dcorn s. Similar ideas already appeared in~\cite[Proposition~19]{BernardiFusy_TrigQuadPent} in the non-bipartite case (and in the dual setting):
\begin{lemm}\label{lemm: Characterization tightness}
Fix $d\geq k >0$.  Let $(\rm,\tau)$ be a pair with $\Delta(\rm)\leq d$, and $\deg(\tau)=k$.

If $\tau$ is black, then: 
\[
	(m,\tau)\text{ is a \dtrum}\quad\Leftrightarrow\quad \alpha_{d,k}^-\text{ is feasible on }(\rm,\tau),
\]
and, in that case, any $\alpha_{d,k}^-$-orientation is accessible.

If $\tau$ is white, then: 
\[
	(m,\tau)\text{ is a \dcorn}\quad\Leftrightarrow\quad \alpha_{d,k}^+\text{ is feasible and quasi-accessible on }(\rm,\tau),
\]
where quasi-accessible means that for any $v\neq \tau$, there exists a forward path from $v$ to $\rho$. 
\end{lemm}

\begin{proof}
Assume that $\tau$ is black, and that $\alpha_{d,k}^-$ is feasible on $(\rm,\tau)$. Let $(R,S)$ be a black cut of $\rm$, and let $w\geq 1$ denote its weight. We aim to prove that $w\geq k$.  Let $d_\bullet$ and $d_\circ$ denote the sum of the degrees of the black and white vertices in $R$, respectively. First, note that since $(R,S)$ is black, we have that $d_\circ = d_\bullet+w$. Then, since  $\alpha_{d,k}^-$ is $(d+1)$-fractional, we have: 
\begin{equation}\label{eq:ori_cut}
	\sum_{v\in R}{\alpha_{d,k}^-(v)} \geq \left(d+1\right) \vert\mathrm{E}_R\vert,
\end{equation}
where $ \vert\mathrm{E}_R\vert$ is the number of edges with both endpoints in $R$. Moreover, on one hand, by definition of $\alpha_{d,k}^-$-orientations, we have: 
\[
	\sum_{v\in R}{\alpha_{d,k}^-(v)} = d\, d_\bullet + d_\circ -k.
\]
On the other hand, since $(R,S)$ is black, we have: 
\[
\left(d+1\right) \vert\mathrm{E}_R\vert	=\left(d+1\right) d_\bullet = d\, d_\bullet + d_\circ - w.
\]
By~\eqref{eq:ori_cut}, it follows that $w\geq k$, and hence that $(\rm,\tau)$ is a \dtrum. \\

Reciprocally, assume that $(\rm,\tau)$ is a \dtrum. To define an $\alpha_{d,k}^-$-orientation on $(\rm,\tau)$, we apply the max-flow min-cut theorem to the following representation of $(\rm,\tau)$ as a flow network.
The \emph{flow network} associated to $\rm$, and denoted by $\tilde{\rm}$, is the directed plane map obtained by replacing each edge $e$ of $\rm$ with $2$ directed copies: one toward its black endpoint, denoted by $e_{\circ\rightarrow\bullet}$, and the other toward its white endpoint, denoted by $e_{\bullet\rightarrow\circ}$. Moreover the \emph{capacities} of the edges are set to be: 
\[
\textrm{cap}(e_{\circ\rightarrow\bullet})\coloneqq 1\quad \text{ and } \quad \textrm{cap}(e_{\bullet\rightarrow\circ})\coloneqq d, \text{ for any }e\in \rE(\rm).
\]
Finally, the \emph{source} of $\tilde\rm$ is $\rho$ and its \emph{sink} is $\tau$. Recall that a \emph{flow} is a function $F:\mathrm{E}(\tilde\rm)\rightarrow \mathbb{Z}_{\geq 0}$ that must satisfy the following conditions:
\begin{itemize}
    \item Capacity constraint: $\forall \vec e\in \mathrm{E}(\tilde \rm)$, $F(\vec e) \leq \textrm{cap}(\vec e)$, 
    \item Conservation of flows: $\forall v \in V\setminus \{ \rho, \tau\}, \sum_{\vec{e}\in\mathrm{E}_{v \rightarrow}}{F (\vec{e})} = \sum_{\vec{e}\in\mathrm{E}_{\rightarrow v}}{F(\vec{e})},$
    \item Source/sink constraint: $\sum_{\vec{e}\in\mathrm{E}_{\rightarrow \rho}}{F(\vec{e})} = \sum_{\vec{e}\in\mathrm{E}_{\tau \rightarrow}}{F(\vec{e})}= 0$,
\end{itemize}
where $\mathrm{E}_{v \rightarrow}$ and $\mathrm{E}_{\rightarrow v}$ denote the set of edges in $\tilde\rm$ directed from $v$ and toward $v$, respectively.
Moreover, the value $|F|$ of the flow is defined by: 
\[
|F| \coloneqq  \sum_{\vec{e}\in\mathrm{E}_{\rho \rightarrow}}{F(\vec{e})} = \sum_{\vec{e}\in\mathrm{E}_{\rightarrow \tau}}{F(\vec{e})}.
\]

The \emph{capacity} of a cut $(R,S)$ of $\tilde\rm$ is the sum of the capacity of the edges of its cut-set that are directed toward a vertex in $S$ if the source is in $R$, or toward a vertex in $R$ if the source is in $S$. We write $\underline{\mathrm{cut}}$ for the minimal capacity of the cuts of $\tilde\rm$, and prove that $\underline{\mathrm{cut}} = k$. Observe that if $(R,S)$ is black (resp. white), then its capacity is equal to its weight (resp. to $d$ times its weight).

Since $\tau$ is black, the capacity of the cut defined by the trivial partition $(V\backslash\{\tau\},{\tau})$ is equal to $k$. Then, let $(R,S)$ be any other cut of $\tilde{\rm}$. Either it corresponds to a black cut of $\rm$, and since $\rm$ is tight, it follows that its capacity is at least $k$. Or, there is at least an edge $\vec{e}$ in the cut-set that is directed toward a white vertex in $S$. Since the capacity of $\vec{e}$ is equal to $d\geq k$, the capacity of $(R,S)$ is at least $k$, which implies that $\underline{\mathrm{cut}} = k$. 

The max-flow min-cut theorem then ensures the existence of a flow $F^\star$ with value $k$, that we use to define an explicit orientation $\cO$ on $\rm$. For every edge $e=\left\{h_\bullet,h_\circ\right\}\in\mathrm{E}(\rm)$, we set:
\begin{equation*}
\begin{aligned}
	\cO(h_\circ) &\coloneqq  1 + F^\star(e_{\bullet\rightarrow\circ}) - F^\star(e_{\circ\rightarrow\bullet})\\
	\cO(h_\bullet) &\coloneqq  d - F^\star(e_{\bullet\rightarrow\circ}) + F^\star(e_{\circ\rightarrow\bullet}).
\end{aligned}
\end{equation*}
First, it follows from the \emph{capacity constraints} that $-1\leq F^\star(e_{\bullet\rightarrow\circ}) - F^\star(e_{\circ\rightarrow\bullet})\leq d$, for every edge $e\in\mathrm{E}(\rm)$. Hence, $\cO$ is a well defined $(d+1)$-fractional orientation. Then, it follows from the conservation of flows at each vertex and from $\vert F^\star\vert = k$ that $\out(v)=\alpha_{d,k}^-(v)$, for every $v\in\mathrm{V}(\rm)$. Thus, $\cO$ is an $\alpha_{d,k}^-$-orientation on $(\rm,\rho,\tau)$. \

To conclude, note that $\cO$ is accessible.
Indeed, for any subset $S\subset \mathrm{V}(\rm)\setminus\{\rho\}$, we have:
\begin{equation*}
\sum_{v\in S}{\alpha_{d,k}^-(v)} \geq \sum_{v\in S}{\alpha_d(v)} > \left(d+1\right) \vert\mathrm{E}_S\vert,
\end{equation*}
where the last inequality follows from the accessibility of the $\alpha_d$-orientations on the bipartite map $\rm$.\\

The case of $\alpha_{d,k}^+$-orientation is similar. The only subtlety is to deal with the uniqueness of the minimal cut (for the converse implication) and with the quasi-accessibility (for  the direct implication). For the uniqueness of the minimal cut, using the quasi-accessibility, we can replace~\eqref{eq:ori_cut} by: 
\[
\sum_{v\in S}\alpha_{d,k}^+(v)  > \left(d+1\right) \vert\mathrm{E}_S\vert,
\]
where $(R,S)$ is any white cut different from the trivial one. The uniqueness of the minimal cut follows from the same arguments as in the previous case. 

For the quasi-accessibility, we define the flow network and the associated orientation in the same way, except that we exchange the role of the source (which becomes $\tau$) and of the sink (which becomes $\rho$), while the definition of $\cO$ remains the same. The same arguments as in the previous case show that the capacity of any cut is at least $k$, with equality only for the trivial cut.

Let $(R,S)\neq (\rV(\rm)\backslash\{\tau\},\{\tau\})$ be a black cut, and let $w$ denote its weight. The following holds by the conservation of flows at each vertex and from $\vert F^\star\vert = k$:
\begin{equation*}
	\sum_{v\in S}\alpha_{d,k}^+(v) = \sum_{v\in S}\sum_{h\sim v} \cO(h) = \left(d+1\right) \vert\mathrm{E}_S\vert + w - \vert F^\star\vert.
\end{equation*}
Since, $(\rm,\rho,\tau)$ is a \dtrum, one has $w>\vert F^\star\vert$, so that:
\begin{equation*}
	\sum_{v\in S}\alpha_{d,k}^+(v) > \left(d+1\right) \vert\mathrm{E}_S\vert.
\end{equation*}
The quasi-accessibility follows. 
\end{proof}

\subsubsection{$\alpha_{d,k}^{+/-}$-blossoming trees } \label{sec:alphadktrees}
We extend the definition of $\alpha_{d,k}^{+/-}$ orientation to blossoming trees (with no additional marked vertex). Fix $0\leq k \leq d$ and let $\rt$ be a blossoming tree with root vertex $\rho$, we define, with a slight abuse of notation, $\alpha_{d,k}^-: \mathrm{V}(\rt)\rightarrow \mathbb{Z}_{\geq 0}$ and $\alpha_{d,k}^+ : \mathrm{V}(\rt)\rightarrow \mathbb{Z}_{\geq 0}$ , by: 
\begin{align}\label{eq: alphadk for trees}
\alpha_{d,k}^-(v) &\coloneqq \alpha_{d}(v) - k\,\mathbf{1}_{\left\{v=\rho\right\}}, \\ \alpha_{d,k}^+(v) &\coloneqq \alpha_{d}(v) + k\,\mathbf{1}_{\left\{v=\rho\right\}}.
\end{align}
Then, similarly to Section~\ref{sub: alpha_d-trees}, an \emph{$\alpha_{d,k}^-$-tree} (resp. an \emph{$\alpha_{d,k}^+$-tree}) is defined as a blossoming \emph{bipartite} tree $\rt$ such that $\Delta(\rt)\leq d$, and equipped with an accessible $\alpha_{d,k}^-$-orientation (resp. \emph{$\alpha_{d,k}^+$-orientation}). Note that both $\alpha_{d,k}^-$-trees and $\alpha_{d,k}^+$-trees correspond to $\alpha_d$-trees when $k=0$. 

For simplicity, we exclude trivial trees consisting of a single black vertex with one closing stem, or a single white vertex with one opening stem, from the class of $\alpha_{1,1}^-$- or $\alpha_{1,1}^+$-trees.

\bigskip

We first start with a lemma that ensures that the closure of $\alpha_{d,k}^-$-trees is bipartite:

\begin{lemm}\label{lemm: Opening and Closing stems constraint on alpha_dk-trees}
Fix $d \geq k\geq 1$. In $\alpha_{d,k}^-$-trees of charge $k$ and in $\alpha_{d,k}^+$-trees of charge $-k$, all opening stems are incident to black vertices, and all closing stems to white vertices. 
\end{lemm}
\begin{proof}
Let $\rt$ be an $\alpha_{d,k}^-$-tree, with $\rc(\rt)=k$. The same line of arguments as in the proof of Lemma~\ref{lemm: Opening and Closing stems constraint on alpha_d-trees} implies the result for non-root vertices. We hence only focus on the root vertex $\rho$. 

If $\rho$ is white, then $\out(\rho) = \alpha_{d,k}^- (\rho) \leq d-1$, so that no opening stem can be incident to $\rho$. 

If $\rho$ is black, let $n_{cl}$, $n_{op}$ and $n$ be its number of adjacent closing stems, opening stems and edges, respectively, and write $(t_1,\ldots,t_n)$ for the sequence of its white sub-planted trees. It follows from the definition of $\alpha_{d,k}^-$ that for any $1\leq i\leq n$, the planted tree $t_i$ is an planted $\alpha_d$-tree. Thus, by definition of $\alpha_{d,k}^-(\rho)$, we have:
\begin{equation}\label{eq:alphadk_rho}
	d\,(n_{cl}+n_{op}+n)-k = \alpha_{d,k}^-(\rho) = (d+1)\,n_{op} + \sum_{j=1}^n{\big(d+1-\mathrm{exc}(\rt_j)\big)}.
\end{equation}
Moreover, Lemma~\ref{lemm: Charge-Excess correspondence} gives:
\[
\sum_{j=1}^n \mathrm{exc}(\rt_j)=\sum_{j=1}^n(\rc(\rt_j)+1)=\rc(\rt)-n_{cl}+n_{op}+n.
\]
Hence, \eqref{eq:alphadk_rho} can be rewritten as:
\begin{equation*}
	d\,(n_{cl}+n_{op}+n)-k = d\,n_{op} + d\,n  + n_{cl}  - k,
\end{equation*}
which implies that $d\,n_{cl} = n_{cl}$. Then, either $d=1$ or $n_{cl}=0$. Since the first case is excluded by definition, it concludes the proof. \

The case of $\alpha_{d,k}^+$-trees is similar and is left to the reader. 
\end{proof}

The proof of Proposition~\ref{prop:bij_charge0} used verbatim in our setting gives:

\begin{lemm} \label{lemm:treesk}
Fix $d\geq k\geq 1$. Up to forgetting the orientation, the set of $\alpha_{d,k}^-$-trees of charge $k$ corresponds to $\mathcal{T}_k^{(d)}$. Similarly, the set of $\alpha_{d,k}^+$-trees of charge $-k$ corresponds to $\mathcal{T}_{-k}^{(d)}$.
\end{lemm}

\subsubsection{Proof of Theorem~\ref{theo: Bijection well-rooted trees with general charge}}

Fix $d \geq k\geq 1$.
Lemma~\ref{lemm: Opening and Closing stems constraint on alpha_dk-trees} ensures that the closure of $\alpha_{d,k}^-$-trees of charge $k$ and of  $\alpha_{d,k}^+$-trees of charge $-k$  are bipartite. In addition, this closure has $k$ unmatched closing incident to its outer face for $\alpha_{d,k}^-$-trees, and $k$ opening stems incident to its outer face for $\alpha_{d,k}^+$-trees. The complete closure is thus naturally endowed with its minimal $\alpha_{d,k}^-$-orientation or $\alpha_{d,k}^+$-orientation.

Therefore, as a consequence of Lemma~\ref{lemm: Characterization tightness}, the complete closure operation is a one-to-one correspondence between:
\begin{itemize}
\item the set of $\alpha_{d,k}^-$-trees of charge $k$, and
\item the set of \dtrum{s} $(\rm,\tau)$, with $\deg(\tau)=k$, and $\Delta(\rm)\leq d$.
\end{itemize}
Similarly, the complete closure operation is a one-to-one correspondence between:
\begin{itemize}
\item the set of $\alpha_{d,k}^+$-trees of charge $-k$, and
\item the set of \dcorn{s} $(\rm,\tau)$, with $\deg(\tau)=k$, and $\Delta(\rm)\leq d$.
\end{itemize}
Moreover, these bijections preserve the vertex degree distribution and the color of the root vertex. Theorem~\ref{theo: Bijection well-rooted trees with general charge} then follows from Lemma~\ref{lemm:treesk}.

Note that, similarly to the case of $\alpha_d$-orientations, we first established a bijection between the sets of \dcorn s and \dtrum s with vertex degree at most $d$  for each $d\geq 1$. Then, we proved that these bijections coincide on their respective domains, as they put these maps in correspondence with well-charged trees, the definition of which is independent of $d$. In conclusion, we obtain a consistent bijection between the entire set of \dcorn~ and \dtrum~ and a set of trees, with no degree constraint.

\begin{rema}
Note that we could also prove directly a property analogous to Property~\ref{claim:stability for orientations}, to describe how minimal $\alpha_{d,k}^-$-orientations (resp. $\alpha_{d,k}^+$-orientations) can be deduced one from another for different values of $d$. Similarly to ~\ref{rem:infinite_bijection}, this ensures that the collections of $\alpha_{d,k}^-$-trees (resp. $\alpha_{d,k}^+$-trees) associated to the same trumpet (resp. cornet) only differ through an additive shift on their orientation, without relying on the correspondence with well-charged trees.
\end{rema}

\section{Enumerative consequences}\label{sec4}

In this section, we provide a new derivation of the enumeration of bipartite plane maps with prescribed vertex degrees, based on Theorem~\ref{theo: Bijection Bipartite maps}. We begin by recalling some results of~\cite{BousquetMelouSchaeffer_Bipartite} concerning the enumeration of well-charged planted trees. We then derive enumerative results for bipartite plane maps, \dtrum s and \dcorn s, and for doubly rooted bipartite maps.

In this section, let $u$ and $\xi$ be two formal variables, and let $\ux\coloneqq \left(x_k\right)_{k\geq 0}$ and $\uy\coloneqq \left(y_k\right)_{y\geq 0}$ denote two families of formal variables.

	\subsection{Well-charged trees}\label{sec: trees gf}
The weight $w^\mathrm{tree}(\rt)$ of a planted well-charged tree $\rt$ is defined by: 
\begin{equation}\label{eq:weight_well_charged_tree}
	w^\mathrm{tree}(\rt)=u^{\#\{\text{opening stems}\}}\prod_{ v\in \rV_\circ(\rt)}x_{\deg(v)}\prod_{v\in \rV_\bullet(\rt)}y_{\deg(v)}.
\end{equation}
Then, for any $k\in\Z$, let $B_k(\ux,\uy, u)$ and $W_k(\ux,\uy, u)$ in $\Q\llbracket\ux,\uy,u \rrbracket$ denote the weighted generating series of well-charged planted trees of charge $k$, rooted at a black or a white vertex, respectively.

We further define the following generating (Laurent) series of well-charged planted trees,  which are elements of $\Q\llbracket\ux,\uy,u, \xi, 1/\xi \rrbracket$:
\begin{equation}
B(u, \xi) \coloneqq  B(\ux,\uy, u, \xi) \coloneqq \sum_{k\leq 1}B_k(\ux,\uy; u) \,\xi^k,
\end{equation}
and,
\begin{equation}
 W(u, \xi) \coloneqq  W(\ux,\uy, u, \xi) \coloneqq \sum_{k\geq 0} W_k(\ux,\uy; u)\, \xi^k.
\end{equation}

Following~\cite{BousquetMelouSchaeffer_Bipartite}, by decomposing a planted well-charged tree at its root, we obtain the recursive relations satisfied by these generating series:
\begin{equation}\label{eq: Recursive eq GF black blossoming trees}
	B(u, \xi) = \left[\xi^{\leq 1}\right] \sum\limits_{l\geq 0}{y_{l+1} \,\big(u\xi^{-1}+W(u, \xi)\big)^l},
\end{equation}
\begin{equation}\label{eq: Recursive eq GF white blossoming trees}
	 W(u, \xi) = \left[\xi^{\geq 0}\right] \sum\limits_{l\geq 0}{x_{l+1}\,(\xi+B(u, \xi))^l},
\end{equation}  
where the formal operators $\left[\xi^{\leq p}\right]$ and $\left[\xi^{\geq p}\right]$ extract the terms of the formal power series for which the exponent of $\xi$ is at most $p$ and at least $p$, respectively. \\

	\subsection{Bipartite maps, \dtrum s and \dcorn s} 

		\subsubsection{Enumeration of bipartite maps}\label{sec: bipartite gf}

For any $\rm \in \mathcal{M}$ and $\bar \rm\in \bar{\mathcal{M}}$, we define their weight as follows:
\begin{align}
	w(\rm)&=u^{\#\mathrm{F}(\rm)}\prod_{v\in \rV_\circ(\rm)}x_{\deg(v)}\prod_{v\in \rV_\bullet(\rm)}y_{\deg(v)},\label{eq:weightPlanarBipartite}\\
	\bar w(\bar \rm)&=u^{\#\mathrm{F}(\bar\rm)-1}\prod_{v\in \rV_\circ(\bar \rm)}x_{\deg(v)}\prod_{v\in \rV_\bullet(\bar\rm)}y_{\deg(v)}   =u^{-1}\,w(\bar\rm).\label{eq:weightPlaneBipartite}
\end{align}
Then $M_\circ(u)\equiv M_\circ(\ux,\uy, u)\in\Q\llbracket\ux,\uy,u \rrbracket$ (resp.  $\bar{M}_\circ(u)\equiv \bar{M}_\circ(\ux,\uy, u)\in\Q\llbracket\ux,\uy,u \rrbracket$) denotes the weighted generating function of white-rooted bipartite planar (resp. plane) maps. It follows directly from the dictionary for generating series \cite[ch.1]{FlajoletSedgewick09}, that 
\begin{equation}\label{eq:dMbar}
    \bar{M}_\circ(u) := \partial_u M_\circ(u).
\end{equation}
Then, Theorem~\ref{theo: Bijection Bipartite maps} implies the following equality of formal power series in  $\Q\llbracket\ux,\uy,u \rrbracket$:
\begin{prop}\label{prop:genSeries_bipartite_wellcharged}
The generating series $\bar M_\circ(u)$ of white-rooted plane bipartite maps satisfies: 
\begin{equation}\label{eq: Mbar(u) in terms of B(u,xi)}
	\bar{M}_\circ(u) = \left[\xi^0\right] \sum\limits_{l\geq 0} x_l \,(\xi + B(u,\xi))^{l}.
\end{equation}
\end{prop}
\begin{proof}
	This is a direct consequence of Theorem~\ref{theo: Bijection Bipartite maps}. Indeed, the definitions of weights given in~\eqref{eq:weight_well_charged_tree} and~\eqref{eq:weightPlaneBipartite} ensure that the bijection between white-rooted bipartite plane maps and white-rooted well-charged trees of charge 0 preserves their weights. 
	The decomposition of these trees at their root vertex into opening stems and black-planted well-charged trees concludes the proof.
\end{proof}		
		
			\subsubsection{Enumeration of \dtrum s and \dcorn s}				
Fix $k\in\mathbb{Z}_{>0}$. Let $M^\mathsf{t}_{\bullet,k}$ and $M^\mathsf{t}_{\circ,k}$ be respectively the weighted generating function of black-rooted and white-rooted \dtrum s, with a marked vertex of degree $k$. Similarly, let $M^\mathsf{c}_{\bullet,k}$ and $M^\mathsf{c}_{\circ,k}$ be respectively the weighted generating function of black-rooted and white-rooted \dcorn s $(\rm,\tau)$, with a marked vertex of degree $k$. 

The weight of a \dtrum~or a \dcorn~$(\rm,\tau)$ is the same as the weight for bipartite planar maps -- defined in~\eqref{eq:weightPlanarBipartite} -- except that no weight is assigned to the marked vertex $\tau$. Then, it follows  from Theorem~\ref{theo: Bijection well-rooted trees with general charge} that: 
\begin{prop}
For any $k>0$,  we have the following equalities of generating series in $\Q\llbracket\ux,\uy,u \rrbracket$:
\begin{align*}
	M^\mathsf{t}_{\bullet,k} &= u^k \left[\xi^k\right] \sum\limits_{l\geq 1} y_l \,(u\xi^{-1} + W(u,\xi))^{l} , &
	M^\mathsf{c}_{\bullet,k} &= \left[\xi^{-k}\right] \sum\limits_{l\geq 1} y_l \,(u\xi^{-1} + W(u,\xi))^{l} ,\\
	M^\mathsf{t}_{\circ,k} &= u^k \left[\xi^k\right] \sum\limits_{l\geq 1} x_l \,(\xi + B(u,\xi))^{l} , &
	M^\mathsf{c}_{\circ,k} &= \left[\xi^{-k}\right] \sum\limits_{l\geq 1} x_l \,(\xi + B(u,\xi))^{l}.
\end{align*}
\end{prop}
\begin{proof}
Since no weight is assigned to the pointed vertex of a trumpet or a cornet, the only subtlety is to keep track of the weights of the faces, when performing the closure of a well-charged tree. 

In the full closure of a well-charged tree with a negative (resp. positive) charge, the number of faces is equal to the number of opening (resp. closing) stems. Since the power of $u$ in the weight of a tree accounts for its number of opening stems, the bijection between cornets and well-charged trees is weight-preserving, and the enumerative formulas for $M^\mathsf{c}_{\bullet,k}$ and $M^\mathsf{c}_{\circ,k}$ follow. For trumpets, the term $u^k$ in the formulas accounts for the fact that there are $k$ more closing stems than opening stems in a well-charged tree with charge $k$.
\end{proof}

	\subsection{Doubly rooted bipartite maps} 

In this paragraph we derive a decomposition of doubly rooted bipartite planar maps, into pairs consisting of a \dtrum~and a \dcorn. In the dual setting, this corresponds precisely to the decomposition of hypermaps with two monochromatic boundaries into pairs consisting of a trumpet and a cornet, as described in~\cite[Section~3.3]{AlbenqueBouttier_slices}. In particular, this leads to recover their enumerative results through a new bijective approach. 

We define a \emph{doubly rooted} bipartite planar map as a triple $(\rm, \rho_1, \rho_2)$ consisting of a bipartite planar map $\rm$ with two marked corners, incident to two different vertices denoted by $\rho_1$ and $\rho_2$, respectively. In this setting, a cut $(R,S)$ is defined as a partition of $\rV(\rm)$, such that $\rho_1\in R$ and $\rho_2\in S$. As before, a cut is said to be \emph{black} (resp. \emph{white}) if all the vertices of $S$ incident to the cut-set are black (resp. white).

Note, that a black cut might not exist. The only case in which it can happen is if $\rho_1$ is black, $\rho_2$ is white and $\rho_1$ and $\rho_2$ are adjacent, see also Corollary~\ref{coro:doubly_rooted} and ~\cite[Lemma~4.6]{AlbenqueBouttier_slices}.

\begin{figure}[t]
  	\centering
  	\includegraphics[width=0.514\linewidth ,page=2]{DoublyRooted.pdf}\qquad
  	\includegraphics[width=0.4\linewidth ,page=1]{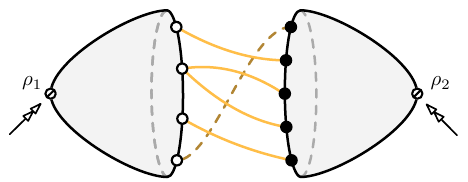}
 	\caption{\label{fig: DoublyRooted} A \dtrum~ $(\rm_1,\tau_1)$ (left), a \dcorn~ $(\rm_2,\tau_2)$ (middle), and one of their gluing into a doubly rooted map $(\rm,\rho_1,\rho_2)$ (right).}
\end{figure}

The following correspondence between doubly rooted maps, \dtrum s and \dcorn s is illustrated in Figure~\ref{fig: DoublyRooted}.
\begin{prop} \label{prop:ktoone}
Let $k\geq 1$. There exists a $k$-to-one correspondence between:
\begin{itemize}
	\item the set of doubly rooted bipartite planar maps with minimal black-cut weight $k$, and
	\item the set of pairs consisting of
	\begin{enumerate}
		\item[1)] a \dtrum~$(\rm_1,\tau_1)$ with $\deg(\tau_1)=k$, and
		\item[2)] a \dcorn~$(\rm_2,\tau_2)$ with $\deg(\tau_2)=k$.
\end{enumerate}	 
\end{itemize}
\end{prop} 
\begin{proof}
Let $(\rm, \rho_1, \rho_2)$  be a doubly rooted bipartite planar map. We denote by $\mathcal{C}_{\bullet}=(R,S)$ the set of all its black cuts with {minimal weight}. We define:
\begin{equation}\label{eq: minimal black cut}
	  S_\mathrm{min}  \coloneqq \bigcap_{(R,S)\in\mathcal{C}_{\bullet}}{S}, \qquad R_\mathrm{min}\coloneqq\mathrm{V}(\rm)\setminus S_\mathrm{min}.
\end{equation}
Note that $(R_1\cup R_2,S_1\cap S_2)\in\mathcal{C}_\bullet$, for any $(R_1,S_1), (R_2,S_2) \in \mathcal{C}_\bullet$. Hence, $(R_\mathrm{min}, S_\mathrm{min})\in\mathcal{C}_\bullet$. We refer to $(R_\mathrm{min}, S_\mathrm{min})$ as the {minimal black cut} of $\rm$ and we denote its (minimal) weight by $k$.

For $A\subset \rV(\rm)$, we denote by $\rE[A]$ the subset of edges of $\rm$ with both extremities in $A$. Then, since $C_{\mathrm{min}}$ has minimal weight, it implies that $(R_\mathrm{min}, \rE[R_\mathrm{min}])$ and $(S_\mathrm{min},\rE[S_\mathrm{min}])$ are connected. For each edge $e$ in the cut set of $C_{\mathrm{min}}$, we cut $e$ into two half-edges $e_r$ and $e_s$, respectively incident to a vertex of $R_{\mathrm{min}}$ and of $S_{\mathrm{min}}$. We obtain two blossoming maps, one rooted at $\rho_1$, and the other one rooted at $\rho_2$. We define $\rm_1$ and $\rm_2$ as the complete closure of these maps, i.e. we add two vertices $\tau_1$ and $\tau_2$ and connect all the half-edges incident to a vertex of $R_{\mathrm{min}}$ (resp. of $S_{\mathrm{min}}$) to $\tau_1$ (resp. to $\tau_2$). The fact that $\rm_1$ and $\rm_2$ satisfy the conditions given in the proposition follows directly from the definition of $C_{\mathrm{min}}$.

Reciprocally, to reconstruct $\rm$ from the pair $(\tilde\rm_1,\tilde\rm_2)$, we first delete $\tau_1$ and $\tau_2$ and transform each edge incident to them into a closing stem (respectively an opening stem). There are exactly $k$ ways to match the $k$ closing stems with the $k$ opening stems, while ensuring that the resulting map is planar, which concludes the proof.
\end{proof}

This result immediately gives the following enumerative corollary: 
\begin{coro}[see also Corollary 4.2 of \cite{AlbenqueBouttier_slices}]\label{coro:doubly_rooted}
Write $M_{\circ,\circ}$, $M_{\circ,\bullet}$ and $M_{\bullet,\bullet}$ for the weighted generating series of doubly rooted bipartite maps, where both root vertices are white (resp. one white and one black, resp. two black). Then,
\begin{align*}
    M_{\circ,\circ}(\ux,\uy, u) &= \sum_{k >0} k \, u^{-k} \, M^\mathsf{t}_{\circ,k} (\ux,\uy, u) \, M^\mathsf{c}_{\circ,k} (\ux,\uy, u),\\
    M_{\circ,\bullet}(\ux,\uy, u) &= \sum_{k >0} k \, u^{-k} \, M^\mathsf{t}_{\circ,k} (\ux,\uy, u) \, M^\mathsf{c}_{\bullet,k} (\ux,\uy, u),\\
    M_{\bullet,\bullet}(\ux,\uy, u) &= \sum_{k >0} k \, u^{-k} \, M^\mathsf{t}_{\bullet,k} (\ux,\uy, u) \, M^\mathsf{c}_{\bullet,k} (\ux,\uy, u).
\end{align*}
\end{coro}

\begin{proof}
    This follows directly from Proposition~\ref{prop:ktoone} once it is established that the maps counted by the generating series of the statement all have black cuts. If $\rho_1$ is white, then $(\{\rho_1\} , \rV \setminus \{\rho_1 \})$ is a black cut. If $\rho_1$ and $\rho_2$ are both black, denoting $R$ as the union of  $\rho_1$ and its neighbors, $(R, \rV \setminus R)$ is a black cut. 
    
   Observe that the decomposition of doubly rooted maps with minimal black-cut weight  $k$ into a pair consisting of a  \dtrum~ and a  \dcorn~ results in the division of each of the  $k$ faces incident to the cut-set into two separate faces. The factor $u^{-k}$ that appears in each sum specifically compensates for this fact.
\end{proof}


\section{Application to the Ising model}\label{sec5}

In this section, we study and enumerate maps equipped with spin configurations. We then derive a Lagrangian parametrization of their generating function in the case of quartic degrees.  

In this section, let $t$, $\nu$ and $u$ represent formal variables, and let $\ux\coloneqq \left(x_k\right)_{k\geq 0}$ and $\uy\coloneqq \left(y_k\right)_{y\geq 0}$ denote two families of formal variables.

	\subsection{Definitions} 
A \emph{spin configuration} on a map $\rm$ is a coloration in black and white of its vertices, i.e. a mapping $\sigma$ from $\rV(\rm)$ to $\{\bullet,\circ\}$. In this context, an edge $e=\{v_1,v_2\}$ is \emph{monochromatic} if $\sigma(v_1)=\sigma(v_2)$, and \emph{frustrated} otherwise. The number of monochromatic edges of  $(\rm,\sigma)$ is denoted by  $m(\rm,\sigma)$, see Figure~\ref{fig: Ising map ex}.

The \emph{weight} $w^{\mathrm{Ising}}(\rm,\sigma)$ of a map endowed with a spin configuration $(\rm,\sigma)$, is defined 
as follows:
\begin{equation}\label{eq:weightIsing}
	w^{\mathrm{Ising}}(\rm,\sigma) \coloneqq t^{\#\mathrm{E}(\rm)}\nu^{m(\rm,\sigma)}u^{\#\mathrm{F}(\rm)}\prod_{v\in \rV_\circ(\rm,\sigma)}x_{\deg(v)}\prod_{v\in \rV_\bullet(\rm,\sigma)}y_{\deg(v)}. 
\end{equation}
Let $\mathcal{I}_\circ$ be the set of rooted planar maps endowed with a spin configuration, in which the root vertex is white. Its associated weighted generating function is the formal power series
defined by:
\begin{equation}\label{eq: def gf I}
	I_\circ(u) \equiv I_\circ(\ux,\uy,t,\nu, u) \coloneqq \sum\limits_{(\rm,\sigma)\in\mathcal{I}_\circ}{w^\mathrm{Ising}(\rm,\sigma)}.
\end{equation}
Note that for a fixed number of edges, there exist only a finite number of planar maps endowed with a spin configuration. Hence, $I_\circ(u)$ lies in $\mathbb{Q}[\ux,\uy,\nu,u]\llbracket t \rrbracket$.

\begin{figure}[t!]
	\centering
	\includegraphics[width=0.3\linewidth ,page=1]{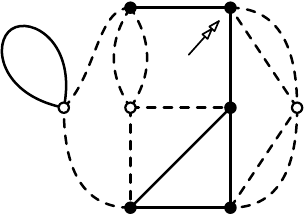}
	\caption{\label{fig: Ising map ex} A black-rooted plane map equipped with a spin configuration, in which all vertices have degree 4. Monochromatic edges are represented by solid lines, while frustrated edges are shown as dashed lines.}
\end{figure}

\begin{rema}
The variable $t$ is redundant in the following sense:
\begin{equation*}
 	I_\circ(\ux,\uy,t,\nu,u) = I_\circ\Big((x_k\, t^{k/2})_{k\geq 1},\, (y_k\, t^{k/2})_{k\geq 1},\, 1,\, \nu ,\, u\Big).
\end{equation*} 
This results from enumerating the edges via the degrees of the vertices in maps. Still, in the rest of this section, the use of the variable $t$ will simplify subsequent discussions.  
\end{rema}

	\subsection{Connection with bipartite maps}\label{sec: Connection with bipartite maps} 
In this section, we recall the well-known correspondence between maps endowed with a spin configuration and bipartite maps; see \cite{Eynard16, BousquetMelouCarranceLouf_Ising, BousquetMelouSchaeffer_Bipartite}, and see Figure~\ref{fig: maps endowed with a spin configuration in D and Bipartite map in L} for an example.

We consider bipartite maps where vertices of degree 2 can either be square or round, while all other vertices are round. The correspondence between maps with a spin configuration and bipartite maps (which is many-to-one) operates as follows: for a given map with a spin configuration, each edge is replaced by a chain, of arbitrary length, made of square vertices of degree two, in such a way that the resulting map is bipartite. Specifically, monochromatic edges are replaced with chains containing an odd number of vertices, whereas frustrated edges are replaced by chains with an even number of vertices.

Reciprocally, for a bipartite map containing square vertices of degree two, the corresponding map with a spin configuration is obtained by contracting each maximal chain of square vertices into a single edge.

In order to translate this correspondence into a relation for generating functions, we first introduce the following notation. Let $\mathcal{M}_{\circ}^{\sq}$ denote the set of bipartite planar maps whose vertices of degree two are square or round and the other vertices are all round and \emph{rooted at a white round vertex}. For $\rm\in \mathcal{M}_{\circ}^{\sq}$, we write $V_\circ(\rm)$ (resp. $V_\bullet(\rm)$) for the set of white (resp. black) \emph{round} vertices of $\rm$, and define the generating function of $\mathcal{M}_{\circ}^{\sq}$ by:
\begin{equation}
    M_\circ^\sq(u)\equiv M_\circ^{\sq}(\ux,\uy, t, \nu, u) := \sum_{\rm \in \mathcal{M}_{\circ}^{\sq}} w^\sq(\rm),
\end{equation}
where the weight $w^\sq(\rm)$ of a map $\rm \in \mathcal{M}_{\circ}^{\sq}$ is given by\textbf{}
\begin{align*}
	w^\sq(\rm)  &\coloneqq t^{\#\{\text{maximal chains of squares}\}}\nu^{\#\{\text{squares}\}}u^{\#\mathrm{F}(\rm)}\prod_{v\in \rV_\circ(\rm)}x_{\deg(v)}\prod_{v\in \rV_\bullet(\rm)}y_{\deg(v)},\\
                &{\ = t^{\#\{\text{maximal chains of squares}\}}\nu^{\#\{\text{squares}\}}\,w(\rm)},\\
                & = t^{\#\mathrm{E}(\rm)} (\nu/t)^{\#\{\text{squares}\}}\,w(\rm).
\end{align*}
Recall that, in the above display, $w$ denotes the weight for bipartite planar maps defined in~\eqref{eq:weightPlanarBipartite}.

Next, we define the following key change of variables that will link the generating functions of bipartite maps and Ising maps:

\begin{defi}
Let $\Theta$ be the change of variables on $\Q\llbracket\ux,\uy, t, \nu, u\rrbracket$ defined by: 
\begin{equation*}
	\Theta\Big(A\big(\ux,\uy, t, \nu, u\big)\Big) \coloneqq A\Big(\ux,\uy, \frac{t}{1-\nu^2}, \nu, u\Big),
\end{equation*}
for any $A\in\Q\llbracket\ux,\uy, t, \nu, u\rrbracket$.
\end{defi}

\begin{figure}[t!]
	\centering
	\includegraphics[width=0.6\linewidth ,page=1]{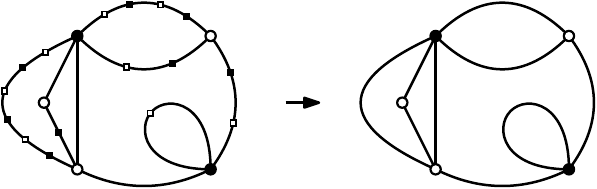}
    \caption{\label{fig: maps endowed with a spin configuration in D and Bipartite map in L} A bipartite map $\tilde{\rm}$ with round vertices of arbitrary degree and square vertices of degree 2 (left). Its weight is ${w(\tilde{\rm})=x_2x_3x_4y_4y_5t^9\nu^{16}u^6}$.
    Its corresponding map endowed with a spin configuration $(\rm,\sigma)$ obtained by contracting the chains of square vertices. Its weight is ${w^\mathrm{Ising}(\rm,\sigma)=x_2x_3x_4y_4y_5t^9\nu^{2}u^6}$ (right).}
\end{figure}

All together, the one-to-many correspondence between bipartite maps and maps endowed with a spin configuration leads to the following:

\begin{clai}\label{claim: Theta Msquare and Ising}
The generating functions $I_\circ(u)$ and $M_\circ^\sq(u)$ satisfy the following relation in $\Q\llbracket\ux,\uy, t, \nu, u\rrbracket$:
\begin{equation}
	M_\circ^\sq(u) = \Theta\Big(I_\circ(u)\Big).
\end{equation}
\end{clai}


In particular, for any $m\geq 3$, the investigation of the generating function of $m$-regular maps endowed with a spin configuration reduces to studying  the generating function of bipartite maps with vertex degrees 2 and $m$, rooted at a vertex of degree $m$. Section \ref{sec: Lagragian Ising Quartic} focuses on this analysis in the quartic case, i.e., when $m=4$.

\begin{rema}
	The change of variables $\Theta$ is invertible on $\Q\llbracket\ux,\uy, t, \nu, u\rrbracket$, indeed we have: 
    \[
    \Theta^{-1}\Big(A\big(\ux,\uy, t, \nu, u\big)\Big) = A\Big(\ux,\uy, t\,(1-\nu^2), \nu, u\Big),
    \]
    so that we can write $\Theta^{-1}(M_\circ^\sq(u))= I_\circ(u)$. 
	However, this change of variables $\Theta^{-1}$ is not combinatorial. By that we mean that, if $A\in \Q\llbracket\ux,\uy, \nu, u\rrbracket$ has non-negative coefficients, then $\Theta^{-1}(A)$ does not necessarily have non-negative coefficients. 

    However, in the next section, we will prove that $\Theta^{-1}(B_1)$ does have non-negative coefficients. 
\end{rema}

	\subsection{Combinatorial interpretation of well-charged trees of charge 1.} \label{sec: Trees charge 1}
In this section, we give an interpretation of the formal power series $\Theta^{-1} \left(B_1(u)\right)$, as the generating function of a family of maps endowed with a spin configuration.

Let $\mathcal{R}$ be the set of quadruples $(\rm,\rho,\tau,\sigma)$ consisting of a rooted planar map endowed with a spin configuration $\sigma$, with a white root vertex $\rho$ of degree one and with a marked black vertex $\tau$ also of degree one. Since $\rho$ and $\tau$ have degree one, we identify them with their unique incident corners. Moreover, we do not require them to be incident to the same face, see Figure~\ref{fig: RPPmaps}.

Let $R(\ux,\uy,t,\nu,u) \in \Q\llbracket\ux,\uy, t, \nu, u\rrbracket$  be the weighted generating function of maps in $\mathcal{R}$, defined as follows:
\begin{equation*}
	R(\ux,\uy,t,\nu,u) \coloneqq  \sum\limits_{(\rm,\tau,\sigma)\in\mathcal{R}}{{(x_1 y_1)}^{-1}\,w^\mathrm{Ising}(\rm,\sigma)},
\end{equation*}
where $w^\mathrm{Ising}$ is defined in~\eqref{eq:weightIsing}. Note that the factor ${1}/{x_1 y_1}$ cancels the weights of the root vertex and of the marked vertex. We have the following:

\begin{prop}\label{prop: Psquare as the generating function of maps endowed with a spin configuration}
Let $P (\ux,\uy , u) \coloneqq u \left( 1 + B_1 (\ux,\uy , u) \right)$ and denote by $P^\sq$ the series in  $\Q\llbracket\ux,\uy, t, \nu, u\rrbracket$ obtained from $P$ by the change of variables:
\begin{equation}\label{eq: Change variables M square}
x_k \mapsto x_k\, t^{k/2} + \nu\,\mathbf{1}_{\left\{k=2\right\}}  \hspace{0.6cm}\text{and,}\hspace{0.6cm} y_k  \mapsto y_k\, t^{k/2} + \nu\,\mathbf{1}_{\left\{k=2\right\}}.
\end{equation}
Then, the following identity holds in $\Q\llbracket\ux,\uy, t, \nu, u\rrbracket$:
\begin{equation}\label{eq: Relation between R tilde and B1 square}
	\Theta\Big(R(\ux,\uy,t,\nu,u)\Big) = t \, P^{\sq}(\ux,\uy,t,\nu,u).
\end{equation}
\end{prop}

\begin{figure}[t]
  	\centering
  	\includegraphics[width=0.3\linewidth ,page=1]{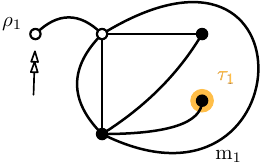}\quad
    \includegraphics[width=0.3\linewidth ,page=2]{RPPmaps.pdf}\quad
  	\includegraphics[width=0.3\linewidth ,page=3]{RPPmaps.pdf}
 	\caption{\label{fig: RPPmaps} Three similar maps: a map endowed with a spin configuration $(\rm_1,\rho_1,\tau_1,\sigma_1)\in\mathcal{R}$ (left), and two bipartite maps $(\rm_2,\rho_2,\tau_2)\in{\mathcal{P}}^{\protect\sq}$ (middle) and $(\rm_3,\rho_3,\tau_3)\in\mathcal{P}$ (right).}
\end{figure}

\begin{proof}
We begin by interpreting the power series $P (\ux,\uy , u)$ as the generating function of the following set of \emph{bipartite} maps, see Figure~\ref{fig: RPPmaps}. 
Let $\mathcal{P}$ denote the set of rooted \emph{bipartite} maps with:
\begin{enumerate}
	\item a white root vertex $\rho$ of degree one, and
	\item a distinguished black vertex $\tau$ of degree one.
\end{enumerate}
Equivalently, $\mathcal{P}$ is the set of \dtrum s rooted at a white vertex of degree one and a black marked vertex of degree one. It follows from Theorem~\ref{theo: Bijection well-rooted trees with general charge}, that it is in bijection with the set of well-charged trees of charge $1$, rooted at a white vertex of degree 1. Excluding the tree with a single white vertex incident to a closing stem, these trees are in bijection with black well-charged planted trees of charge 1. 
Hence, the generating series of $\mathcal{P}$ (with no weight given to $\tau$ and $\rho$) is equal to the formal power series $P (\ux,\uy , u)$. 
Equivalently, we have:
 \begin{equation*}
	P(\ux,\uy,u) = \sum\limits_{\rm\in\mathcal{P}}{{(x_1 y_1)}^{-1}\,w(\rm)} .
\end{equation*}

Now, let $\mathcal{P}^\sq$ be the set of rooted \emph{bipartite} maps with :
\begin{enumerate}
    \item round and square vertices of degree $2$,
    \item round vertices of degree different that $2$,
	\item a white (round) root vertex $\rho$ of degree one,
	\item a black (round) marked vertex $\tau$ of degree one.
\end{enumerate}
Giving a weight $(x_1y_1t)^{-1}w^\sq(\rm)$ to any $\rm \in \mathcal{P}^\sq$, we obtain that the generating series of $\mathcal{P}^\sq$ is equal to the series ${P}^\sq$ defined in the proposition.



Finally, as explained in Section~\ref{sec: Connection with bipartite maps}, the operation of contracting the chains of squares is a many-to-one correspondence between the maps in $\mathcal{P}^\sq$ and the maps in $\mathcal{R}$. In particular, Equation \eqref{eq: Relation between R tilde and B1 square}  is now just a specialization of Claim \ref{claim: Theta Msquare and Ising}.
\end{proof}

\section{Explicit computations in the quartic case}\label{sec6}
In this section, we apply the previous enumerative results to quartic maps, i.e. maps for which every vertex has degree $4$. We first provide explicit parametrizations for the generating functions of bipartite maps of degree $2$ and $4$, then derive explicit expressions for quartic maps with an Ising model. All computations performed in this section are detailed in the Maple companion~\cite{Maple}.

	\subsection{Bipartite maps with vertex degrees $2$ and $4$}\label{sec: Bib Quartic}
In this section, we set $x_k=y_k=0$ for all $k\notin\left\{2,4\right\}$, then we have: 
\begin{prop}\label{prop: Algebraic parametrization for M(u)}
Let $P$ be the unique formal power series in $\Q\left[u\right]\llbracket x_2,x_4, y_2,y_4\rrbracket$ with constant term $u$ satisfying
\begin{equation}\label{eq: Algebraic equation for P}
	P = u + 3x_4y_4 P^3 + P \frac{\big(x_2 + 3x_4y_2 P\big)\big(y_2 + 3x_2y_4 P\big)}{\big(1-9x_4y_4 P^2\big)^2}.
\end{equation}
The generating function of white-rooted bipartite planar maps $M_\circ(u)$ can be expressed in terms of $P$ as follows:
\begin{equation}
	M_\circ(u) = 15(x_4y_4)^2P^6 - 5x_4y_4P^4 + 4ux_4y_4P^3 + \frac{1}{3}(x_2y_2-1)\, P^2 +  \frac{4}{3}uP - u^2.
\end{equation}
\end{prop}
\begin{proof}
We first compute the generating series of well-charged trees.
Note that, due to the vertex degree constraint, the only charges allowed are equal to $1$ or $3$ for white vertices, and to $-3, -1$ or $1$ for black vertices. The recursive equations \eqref{eq: Recursive eq GF black blossoming trees} and \eqref{eq: Recursive eq GF white blossoming trees} can be rewritten as the following system of equations, obtained in~\cite[Sec. 6.1]{BousquetMelouSchaeffer_Bipartite}:
\begin{equation*}
\centering
\begin{cases}
B_{-3}(u) & \hspace{-0.26cm}=  y_4 u^3,\\
B_{-1}(u) & \hspace{-0.26cm}=  y_2 u + 3 y_4 u^2 \, W_1(u),  \\
B_1(u)	  & \hspace{-0.26cm}= y_2 \, W_1(u) + 3y_4\,\Big(u^2 W_3(u) + u W_1(u)^2 \Big),\\
W_1(u)	  & \hspace{-0.26cm}=   x_2 \, \Big(1+B_{1}(u)\Big) + 3 x_4 \, B_{-1}(u) \, \Big(1+B_{1}(u) \Big)^2, \\
W_3(u)	  & \hspace{-0.26cm}=   x_4 \, \Big(1+B_{1}(u) \Big)^3.
\end{cases}
\end{equation*}
As already established in~\cite[Eq.(7) and (8)]{BousquetMelouSchaeffer_Bipartite}, the formal power series $P:=P(\ux,\uy,\nu,u) = u\, (1 + B_1(u))$ is the unique formal power series with constant coefficient equal to $u$ and satisfying \eqref{eq: Algebraic equation for P}. Moreover, all the generating series of well-charged trees can be expressed in terms of $P$, and we have:
\begin{equation*}
B_{-3}(u)=  y_4 u^3, \qquad
B_{-1}(u)=  u\frac{y_2+3x_2y_4P}{1-9x_4y_4P^2},   \qquad
B_1(u)	 = \frac{P}{u}-1,
\end{equation*}
and 
\begin{equation*}
W_1(u)= \frac{P(x_2+3y_2x_4P)}{u(1-9x_4y_4P^2)}, \qquad
W_3(u)=   x_4 \frac{P^3}{u^3}.
\end{equation*}

From now on, our approach differs from the one in~\cite{BousquetMelouSchaeffer_Bipartite}: we apply Theorem~\ref{theo: Bijection Bipartite maps} to express the generating series $\bar{M}_\circ(u)$ of plane rooted bipartite maps in terms of the generating series of blossoming trees (instead of considering \emph{balanced} blossoming trees). Applying Proposition~\ref{prop:genSeries_bipartite_wellcharged}, we get:
\begin{equation}\label{eq: Param barM with P, quadric case}
	\bar{M}_\circ(u) = 2x_2 P\frac{y_2 + 3x_2y_4P}{1-9x_4y_4P^2} +  2x_4P^2\left(y_4P + 3\left(\frac{y_2 + 3x_2y_4P }{1-9x_4y_4P^2}\right)^2\right).
\end{equation}
By~\eqref{eq:dMbar} and since $M_\circ(0)=0$, the series $M_\circ(u)$ can then be expressed as the following formal integral:
\begin{equation}
	M_\circ(u) = \int_0^u \bar{M}_\circ(v)dv = \int_0^u \frac{\bar{M}_\circ(\ux,\uy,v)}{\partial_vP(\ux,\uy,\nu,v)} \partial_vP(\ux,\uy,v)dv.
\end{equation}
Next, we prove that 
\begin{equation}\label{eq:rational}
\frac{\bar{M}_\circ(\ux,\uy, u)}{\partial_uP(\ux,\uy,u)} = \mathrm{Rat}(\ux,\uy,P(\ux,\uy,u)),
\end{equation}
where $\mathrm{Rat}$ is a rational function. Indeed, from~\eqref{eq: Param barM with P, quadric case}, we know that $\bar{M}_\circ$ is a rational function in $\ux$, $\uy$ and $P$. Then, differentiating~\eqref{eq: Algebraic equation for P} with respect to $u$, we get that there exist two polynomials $Q_1,Q_2\in \mathbb{Q}[\ux,\uy,p]$ such that: 
\[
\partial_u P(\ux,\uy,u) \cdot Q_1(\ux,\uy,P(\ux,\uy,u))= Q_2(\ux,\uy,P(\ux,\uy,u)),
\]
and~\eqref{eq:rational} follows.

By an elementary change of variables, and using that $P(\ux,\uy,0)=0$, we obtain: 
\[
M_\circ(u) = \int_0^u \mathrm{Rat}(\ux,\uy,P(\ux,\uy,v))\partial_vP(\ux,\uy,v)dv = \int_0^{P(\ux,\uy,u)} \mathrm{Rat}(p) dp.
\]

The integral can be computed explicitly, and it turns out that it still belongs to $\mathbb{Q}(x_2,x_4,y_2,y_4,P(u))$, see the Maple companion file~\cite{Maple}. We can further simplify this expression using \eqref{eq: Algebraic equation for P}, which concludes the proof. 
\end{proof}

Following exactly the same chain of arguments, we can specialize this result to maps rooted at a vertex of degree 4. The only difference is to replace~\eqref{eq: Param barM with P, quadric case} by the following appropriate expression for the generating series $\bar{M}_{\circ,4}$ of bipartite plane maps rooted on a white vertex of degree 4:
\begin{equation}\label{eq: Param barM with P, quadric case, root degree 4}
	\bar{M}_{\circ,4}(u) =  2x_4P^2\left(y_4P + 3\left(\frac{y_2 + 3x_2y_4P }{1-9x_4y_4P^2}\right)^2\right).
\end{equation}
We obtain:
\begin{prop}[see also Proposition 22 of \cite{BousquetMelouSchaeffer_Bipartite}]\label{prop: Algebraic parametrization for M4(u)}
The generating function of rooted bipartite planar maps with a white root vertex of degree 4 -- denoted by  $M_{\circ,4}$ -- can be expressed in terms of $P$ as follows:
\begin{equation}
	M_{\circ,4}(x_2,x_4,y_2,y_4,u) = \frac{\mathrm{Pol}(x_2,x_4,y_2,y_4,u,P)}{9 x_{4} \left(9 P^{2} x_{4} y_{4}-1\right)}.
\end{equation}
where $\mathrm{Pol}(x_2,x_4,y_2,y_4,u,p)\in\mathbb{Q}\left[x_2,x_4,y_2,y_4,u,p\right]$ is defined as follows: 
\begin{align*}
&\mathrm{Pol}(x_2,x_4,y_2,y_4,u,p) \\
&\quad\coloneqq  1215 {x_{4}}^{4} {y_{4}}^{3}  p^{8}		-540  {x_{4}}^{3} {y_{4}}^{2} p^{6} 	+27 {x_{4}}^{2} {y_{4}}^{2} \left(12 x_{4}u-{x_{2}}^{2}\right) p^{5}		+18 {x_{4}}^{2} y_{4} \left(1-x_{2} y_{2}\right) p^{4}\\
&\qquad	 +6 x_{4} y_{4} \left(12 x_{4}u-5 {x_{2}}^{2}\right) p^{3} 	+\left(45 {x_{2}}^{2} x_{4} y_{4}u +3 x_{4} -3 {x_{2}}^{4} y_{4}-6 x_{4} x_{2} y_{2}-81 x_{4}^{2} y_{4}u^2\right) p^{2}	\\
&\qquad	+\left(1-x_{2} y_{2}\right) \left({x_{2}}^{2}-12 x_{4}u\right) p+ u\left(9 x_{4}u-{x_{2}}^{2}\right).
\end{align*}
\end{prop}
Note that this parametrization differs from the one in \cite{BousquetMelouSchaeffer_Bipartite}, which was obtained using a different ``unpointing''. Still both  are equivalent, and one can be derived from the other by using~\eqref{eq: Algebraic equation for P}.

	\subsection{The Ising model in the quartic case}\label{sec: Lagragian Ising Quartic}
In this section, we deal with white-rooted quartic planar maps endowed with a spin configuration, i.e. we set $x_k=y_k=0$ for all $k\neq 4$. To simplify notation, we write $x=x_4$, $y=y_4$ and $I_\circ(x,y,t,\nu,u)=I_\circ(\ux,\uy,t,\nu,u)$ with a slight abuse of notation. 
We establish the following explicit Lagrangian rational parametrization for $I_\circ(u)$ in terms of a new formal power series with \emph{non-negative} coefficients:

\begin{theo}\label{theo: Lagragian parametrization Ising GF}
The generating function of white-rooted quartic planar maps endowed with a spin configuration is given by:
\begin{equation}\label{eq: I(u) in terms of S(U)}
	I_\circ(x,y,t,\nu,u)=\frac{\mathrm{Pol}_{I_\circ}(x,y,t,\nu,u,Q)}{9\left(1-\nu^{2}\right)t^{4} \Big(1+3x\left(1-\nu^{2}\right)Q\Big)},
\end{equation}
where the series $Q\equiv Q(x,y, t,\nu,u)$ is the unique formal power series in $\Q(x,y,\nu,u)\llbracket t \rrbracket$ with constant term 0 that satisfies the Lagrangian equation:
\footnotesize
\begin{equation}\label{eq: Lagrangian equation for S(U)}
	t^2=\frac{Q  \left(1-3 \nu^{2} \left(x +y \right) Q -3 x y \left(1-\nu^{2}\right) \left(3 \nu^{2}+7\right) Q^{2}+135 x^{2} y^{2} \left(1-\nu^{2}\right)^{3} Q^{4}-243 x^{3} y^{3} \left(1-\nu^{2}\right)^{5} Q^{6}\right)}{u \, {\left(1-9 x y \left(1-\nu^{2}\right)^{2} Q^{2}\right)}^{2}},
\end{equation}
\normalsize
and where the polynomial $\mathrm{Pol}_{I_\circ}(x,y,t,\nu,u,q)\in\mathbb{Q}\left[x,y,t,\nu,u,q\right]$ is defined as follows: 
\scriptsize
\begin{align*}
&\mathrm{Pol}_{I_\circ}(x,y,t,\nu,u,q) \\
&\quad\coloneqq 405x^{3} y^{2}  \left(1-\nu^{2}\right)^{4}q^{7}+351x^{2} y^{2} \left(1-\nu^{2}\right)^{3} q^{6} +27xy\left(1-\nu^{2}\right)^{2}  \left(\nu^{2} y-\left(5+12 t^{2} u \left(1-\nu^2 \right)^{2} y\right) x\right) q^{5} \\
&\quad +3xy\left(1-\nu^{2}\right)\left(36 t^{2} u  \left(1-\nu^{2}\right)^{2}x-3\nu^{2}-47\right) q^{4}+\left(\left(252 t^{2} u  \left(1-\nu^{2}\right)^{2}y-6\nu^{2}-9\right) x-15 \nu^{2} y\right) q^{3}\\
&\quad +\left(\left(36t^{2} u  \left(1-\nu^{2}\right)-108t^{4} u^{2} y  \left(1-\nu^{2}\right)^{3}\right) x+5+9\nu^{2} t^{2} u  \left(1-\nu^{2}\right)y\right) q^{2}-t^{2} u \left(27 t^{2} u  \left(1-\nu^{2}\right)^{2}x-3\nu^{2}+8\right) q\\
&\quad +3t^{4} u^{2} \left(1-\nu^{2}\right).
\end{align*}
\normalsize
Moreover, the series $Q$ has non-negative integer coefficients. 
\end{theo}

\begin{proof}
All the computations performed for this proof are available in the Maple companion file~\cite{Maple}. 
We first observe that, in this setting, $M_\circ^\sq(\ux,\uy,t,\nu,u)=M_{\circ,4}(\nu,xt^2,\nu,yt^2,u)$. By Claim \ref{claim: Theta Msquare and Ising}, we obtain the following equality of generating series: 
\begin{equation*}
	I_\circ(x,y,t,\nu,u)=\Theta^{-1}\Big(M_{\circ,4}(\nu,xt^2,\nu,yt^2,u)\Big).
\end{equation*}
We can then apply the change of variables $\Theta^{-1}$ to the expression of $M_{\circ,4}$ given in Proposition~\ref{prop: Algebraic parametrization for M4(u)}, to obtain a rational expression for $I_\circ$ in terms of $x,y,\nu$ and $Q$, where $Q$ is defined by: 
\begin{equation*}
	Q \coloneqq t \, \Theta^{-1}\Big(t \, P (\nu,xt^2,\nu,yt^2, u)\Big) = t \, \Theta^{-1}\Big(t \, P^\sq (x, y ,t, \nu , u)\Big).
\end{equation*}
The fact that $Q$ satisfies Equation~\eqref{eq: Lagrangian equation for S(U)} follows directly from applying the change of variables $\Theta^{-1}$ in the equation verified by $P$ given in Proposition~\ref{prop: Algebraic parametrization for M(u)}. Finally Proposition~\ref{prop: Psquare as the generating function of maps endowed with a spin configuration} establishes that
\begin{equation*}
	Q = t\, R(\ux,\uy,t,\nu,u),
\end{equation*}
with the specialization $x_k=y_k=0$ for all $k\neq 4$. Thus, $Q$ is also the generating function of a family of Ising maps, and its coefficients are non-negative integers.
\end{proof}

\appendix

\section{Geodesic properties of \texorpdfstring{$\alpha_d$}{alpha-d}-orientations and connection with Eulerian orientations}\label{AppA}
In this appendix, we further explore the properties of $\alpha_d$-orientations on bipartite maps. We begin by examining their connection to the dual directed geodesic labeling, and then we show how they generalize the well-known Eulerian and quasi-Eulerian orientations.

\subsection{Geodesic properties of \texorpdfstring{$\alpha_d$}{alpha-d}-orientations}
		
\subsubsection{Directed geodesic labeling}\label{sub:directed_geodesic_labeling}
A \emph{directed} planar map is a planar map in which every edge has a direction, in other words it is the planar embedding of a \emph{directed} graph. Let $\rm$ be a directed map that is strongly connected, meaning that there exists a directed path between every pair of vertices. We can then define a \emph{directed quasi-distance} on $\rm$ as follows: For two vertices $u,v\in \mathrm{V}(\rm)$, the directed distance -- denoted by $\vec{d}(u,v)$ --  from $u$ to $v$ is defined as the length of a shortest directed path from $u$ to $v$. 

Additionally if $\rm$ is pointed at $v_\star$, then the \emph{(directed) geodesic labeling} $\ell : \mathrm{V}(\rm) \rightarrow \mathbb{Z}_{\geq 0}$ is defined by $\ell(v)=\vec{d}(v_\star,v)$ for any $v\in \mathrm{V}(\rm)$. The geodesic labeling also admits the following characterization: 
\begin{prop}\label{prop: Characterization geodesic labeling}
Let $(\rm,v_\star)$ be a pointed directed planar map. Consider a labeling of its vertices ${\ell:\mathrm{V}(\rm)\rightarrow\Z_{\geq 0}}$, such that:
\begin{enumerate}[1.]
	\item $\ell(v_\star)=0$,
	\item for any $v\in \mathrm{V}(\rm)\backslash\{v_\star\}$, there exists a directed edge from a vertex $u$ to $v$ in $\rm$ such that ${\ell(v)=\ell(u)+1}$,
	\item for any directed edge from $u$ to $v$ in $\rm$, $\ell(v) \leq \ell(u) +1$.
\end{enumerate}
Then $\ell$ is the directed geodesic labeling of $\rm$.
\end{prop}
\begin{proof}
	First, the geodesic labeling on $\rm$ clearly satisfies conditions 1, 2 and 3. Reciprocally, $2.$ and $3.$ together imply that, for any $v\in \mathrm{V}(\rm)\backslash \{v_\star\}$, we have:
	\[
		\ell(v) = \min \{\ell(u)+1, \text{ where }u\text{ is such that }\overrightarrow{uv} \in \rm\}.
	\]
	Then, the conclusion follows by an immediate induction on the value of $\ell(v)$.
\end{proof}

	\subsubsection{Geodesic properties of \texorpdfstring{$\alpha_d$}{alpha-d}-orientations}
In this section, we prove that the minimal $\alpha_d$-orientations of a bipartite map are deeply related to the geodesic labeling on its dual (Eulerian) map, see Figure~\ref{fig: Correspondence alpha-orientation and geodesic labeling} for an illustration. 

\begin{prop}\label{prop: Correspondence alpha-orientation and geodesic labeling}
Let $\rm$ be a rooted bipartite plane map.
Fix $d\geq\Delta_\circ(\rm)$ and let $\cO$ be its minimal $\alpha_d$-orientation. 

Let $(\rm^\dagger,v_\star^\dagger)$ be the dual map of $\rm$, which is pointed at the vertex $v_\star^\dagger$ corresponding to the outer face of $\rm$. Recall from Section~\ref{sub:duality} that the edges of $\rm^\dagger$ are canonically oriented by requiring that white faces are on their left, and write $\ell:\rV(\rm^\dagger)\rightarrow\mathbb{Z}_{\geq 0}$ for the directed geodesic labeling on $(\rm^\dagger,v_\star^\dagger)$.

Then, for any directed edge $e^\dagger=(v_1^\dagger, v_2^\dagger)\in \mathrm{E}(\rm^\dagger)$, we have:
\begin{equation}\label{eq: Correspondence alpha-orientation and geodesic labeling}
  \ell(v_1^\dagger)-\ell(v_2^\dagger) = \cO_\circ(e) - 1,
\end{equation}  
where $e\in \mathrm{E}(\rm)$ is the dual edge of $e^\dagger$, and where $\cO_\circ(e)$ stands for the orientation of the half-edge of $e$ incident to its white endpoint.
\end{prop}

\begin{figure}
  	\centering
  	\includegraphics[width=0.4\linewidth ,page=5]{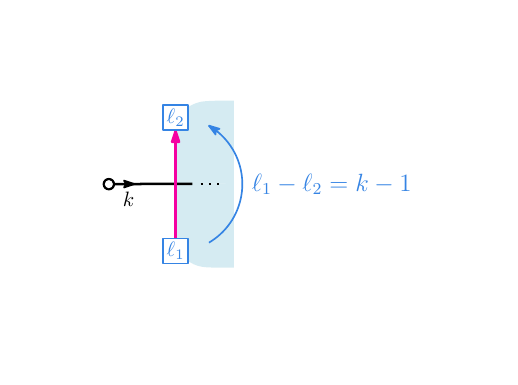}
  	\includegraphics[width=0.4\linewidth ,page=6]{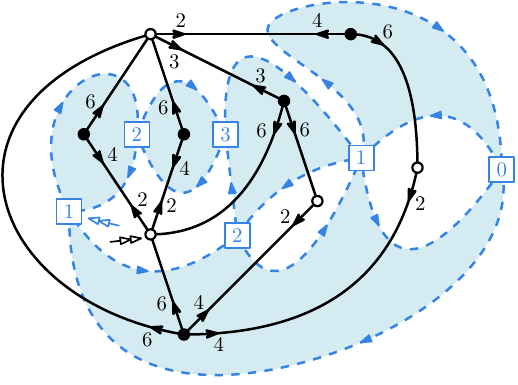}
 	\caption{\label{fig: Correspondence alpha-orientation and geodesic labeling} Correspondence between the minimal $\alpha_d$ orientation and the dual geodesic labeling (left), and an example of a rooted bipartite plane map - drawn in black - and of its dual rooted pointed Eulerian map - drawn in blue (right).}
\end{figure}

\begin{rema}
This proposition is closely related to~\cite[Lemma~18]{BernardiFusy_Hypermaps}, where the authors introduced a family of orientations on Eulerian maps (called hypermaps), and prove a characterization of these orientations similar to Proposition~\ref{prop: Characterization geodesic labeling}, but in the dual setting. 
\end{rema}

\begin{rema}\label{rem:proofclaim}
Note that Proposition~\ref{prop: Characterization geodesic labeling} gives another proof of Property~\ref{claim:stability for orientations}. In fact, for minimal orientations, the value of the orientation of half-edges incident to white vertices does not depend on $d$, as long as $d\geq \Delta_\circ$; and the value of the orientation of half-edges incident to black vertices can be deduced immediately.
\end{rema}


\begin{proof}[Proof of Proposition~\ref{prop: Correspondence alpha-orientation and geodesic labeling}]
The proof of the proposition proceeds in two steps. First, we prove that the conditions given in~\eqref{eq: Correspondence alpha-orientation and geodesic labeling} are consistent; i.e. that there exists a labeling $\tilde{\ell}$ that verifies all of them (and such that $\tilde{\ell}(v_\star^\dagger)=0$). Second, we show that $\tilde{\ell}$ satisfies all the conditions of Proposition~\ref{prop: Characterization geodesic labeling}, and thus coincides with the dual geodesic labeling. 

To address the first point, fix a spanning tree $\rt$ of $\rm^\dagger$. Set $\tilde{\ell}(v_\star^\dagger) = 0$, and extend the labeling to each vertex $v^\dagger \in \rV(\rm^\dagger)$ by successively applying Equation~\eqref{eq: Correspondence alpha-orientation and geodesic labeling} along the path from the pointed vertex $v_\star^\dagger$ to $v^\dagger$ in the spanning tree $\rt$. We can prove that: 

\begin{lemm}\label{lem:labeling_independent}
The labeling $\tilde{\ell}$ obtained does not depend on the choice of $\rt$.
\end{lemm}

The proof of the lemma is postponed to the end of the proof of the Proposition. Finally, we prove that $\tilde \ell$ is the geodesic labeling of $\rm^\dagger$, by applying the characterization of Proposition~\ref{prop: Characterization geodesic labeling}. The first condition holds immediately by construction. To verify the second and third conditions, let $e^\dagger:=(v_1^\dagger,v_2^\dagger)\in E(\rm^\dagger)$ as in the proposition. Since $\cO_\circ(e)\geq 0$, we have:
\begin{equation}\label{eq:variation_geodesic_labeling}  
\tilde{\ell}(v_2^\dagger) = \tilde{\ell}(v_1^\dagger) - \Big(\cO_\circ(e)-1\Big) \leq  \tilde{\ell}(v_1^\dagger)+1.
\end{equation}
This completes the verification of the third condition. To prove the second condition, observe that the inequality in~\eqref{eq:variation_geodesic_labeling} is an equality if and only if $e$ is saturated from its black endpoint to its white endpoint. Let $f_2$ be the face of $\rm$ dual to $v_2^\dagger$. Since $\cO$ is minimal, there exists a saturated edge $\tilde e\in \mathrm{E}(\rm)$ incident to $f_2$, and such that $f_2$ lies on its right. By Property~\ref{claim: Saturated edges are from black to white endpoints}, $\tilde e$ is saturated from its black endpoint to its white endpoint. This concludes the proof.
\end{proof}	

\begin{figure}[t]
  \centering
  \begin{subfigure}[t]{0.3\linewidth}
    \centering
    \includegraphics[width=\linewidth ,page=4]{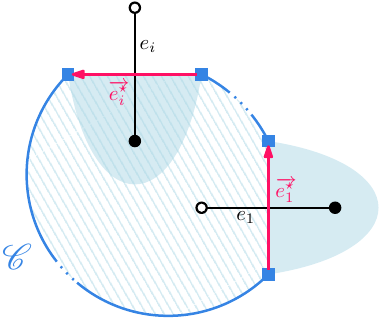}
    \caption{Counterclockwise cycle surrounding several faces, and}
  \end{subfigure}\qquad
  \begin{subfigure}[t]{0.6\linewidth}
    \centering
    \includegraphics[width=\linewidth, page=3]{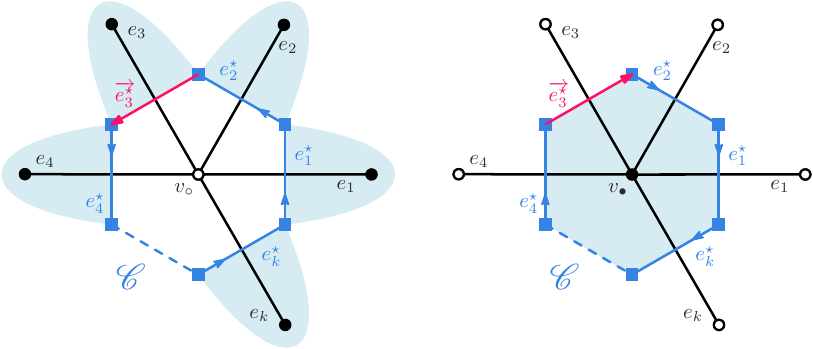}
    \caption{directed contours of a white and black face, respectively.}
  \end{subfigure}
  \caption{\label{fig: labelings on cycle to prove the correspondence} A general cycle where  $\varepsilon(\ora{e_1^\dagger})=1$ and $\varepsilon(\ora{e_i^\dagger})=-1$, with the dashed area indicating the interior of $\mathscr{C}$ (left), and cycles being the direct contour of white and black face, respectively, where $\varepsilon(\ora{e_i^\dagger})$ is constantly $1$ (right).}
\end{figure}

We end this section by proving Lemma~\ref{lem:labeling_independent}.
\begin{proof}[Proof of Lemma~\ref{lem:labeling_independent}]
The proof is illustrated in Figure~\ref{fig: labelings on cycle to prove the correspondence}. Let $\ora{e^\dagger}$ be a directed edge of $\rm^\dagger$. We set:
\begin{equation*}
    \varepsilon(\ora{e^\dagger}) \coloneqq
    \begin{cases*}
      +1  & 
      if the direction of $\ora{e^\dagger}$ coincides with the canonical direction of $e^\dagger$ on $\rm^\dagger$ \\
      -1  & otherwise.
    \end{cases*}
\end{equation*}
Let $\mathscr{C}:=(\ora{e_1^\dagger},\ldots,\ora{e_k^\dagger})$ be a simple directed cycle in $\rm^\dagger$. To prove the lemma, it suffices to show that the total variation of the labels along $\mathscr{C}$, as defined by~\eqref{eq: Correspondence alpha-orientation and geodesic labeling}, is equal to zero; or, more formally that:
\begin{equation}\label{eq:contour_cycle}
\sum_{i=1}^{k} \varepsilon(\ora{e_i^\dagger})\, \big( \cO_\circ(e_i) - 1 \big) = 0,
\end{equation}
where $e_i$ stands for the dual edge of $e_i^\dagger$ for any $i\in\left\{1,\ldots, k\right\}$.\\

The cycle $\mathscr{C}$ disconnects $\rm^\dagger$ into 2 connected components composed of faces. The \emph{interior} of $\mathscr{C}$ is defined as the one which does not contain $v_\star^\dagger$. Since the contour of $\mathscr{C}$ can be written as the sum of the contours of the faces in its interior, it is in fact enough to establish~\eqref{eq:contour_cycle} when $\mathscr{C}$ is the contour of a face. 

Assume that $\mathscr{C}$ is the directed contour of a white face of $\rm^\dagger$, so that $\varepsilon(\ora{e_i^{\dagger}})=1$, for any $1\leq i\leq k$. Write $v_\circ$ for the vertex of $\rm$ associated with this face.
By definition of $\alpha_d$-orientations, we have:
\[
  \sum_{i=1}^{k} \Big(\cO_\circ(e_i) - 1 \Big) = \sum_{h\sim v_\circ}{\cO(h)} - \deg(v_\circ) = \alpha_d(v_\circ) - \deg(v_\circ) = 0.
\]
The case of a black face is similar and is left to the reader.
\end{proof}

	\subsection{Connection between \texorpdfstring{$\alpha_d$}{alpha}-orientations and Eulerian orientations}
	
In this section, we show how $\alpha_d$-orientations generalize the classical notions of Eulerian (or more precisely quasi-Eulerian) orientations to bipartite maps. Unlike in the rest of the article, here, $\rm$ refers to a general map, which is not necessarily bipartite.\

An \emph{Eulerian orientation} is a $1$-fractional $\alpha$-orientation, where ${\alpha(\cdot)\coloneqq \deg(\cdot)}/2$. 
If $\rm$ is Eulerian, i.e. if all its vertices have even degree, it is classical and follows directly from Euler's theorem, that $\rm$ can be endowed with an Eulerian orientation, and that any Eulerian orientation is accessible.\
More generally, a \emph{quasi-Eulerian orientation} is a $2$-fractional $\alpha$-orientation with ${\alpha(\cdot)\coloneqq \deg(\cdot)}$. It is straightforward to see that any planar map can be endowed with a quasi-Eulerian orientation, and that any of them is accessible. Moreover, by Proposition~\ref{prop: Unique minimal orientation}, every Eulerian planar map admits a unique minimal accessible Eulerian orientation, and every planar map admits a unique minimal accessible quasi-Eulerian orientation.

\begin{figure}[t]
	\centering
	\includegraphics[width=0.45\linewidth ,page=2]{QuasiEulerianOrientation.pdf}\qquad
	\includegraphics[width=0.4\linewidth ,page=1]{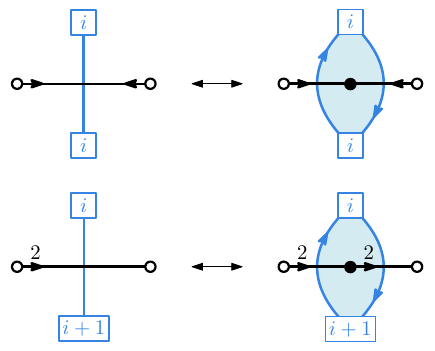}
	\caption{\label{fig: Quasi-Eulerian to alpha d-orientation} Quasi-Eulerian orientations. A plane map with its minimal quasi-Eulerian orientation and dual undirected geodesic labeling (left), the possible orientations of an edge, its connection to the dual labeling, and the corresponding hypermap with black faces of degree 2 (right).}
\end{figure}

Minimal Eulerian and quasi-Eulerian orientations are known to be deeply related to the geodesic labeling on $\rm^\dagger$, where the geodesic labeling is defined here for the classical notion of graph distance on $\rm^\dagger$, and not on its directed definition as in Section~\ref{sub:directed_geodesic_labeling}. More precisely, as illustrated on Figure~\ref{fig: Quasi-Eulerian to alpha d-orientation}, in the minimal orientation: 
\begin{itemize}
 \item non-saturated edges separate faces with equal geodesic labeling, 
 \item saturated edges separate faces with different geodesic labeling, such that the face with higher label is on the right of the saturated edge. Note that, in that case, both labels differ exactly by 1, since we consider the non-directed geodesic labeling. 
\end{itemize}
\smallskip

This fact can be recovered as a special case of Proposition~\ref{prop: Correspondence alpha-orientation and geodesic labeling}. Indeed, any planar map can be transformed into a bipartite map by inserting a black vertex (of degree 2) in the middle of each of its edges. Dually, this corresponds to the usual transformation of planar maps into hypermaps by replacing every edge by a face of degree 2. Observe that the directed geodesic labeling on the hypermap obtained coincides with the undirected geodesic labeling on the initial map. 

Let $\rm$ be a rooted planar map,  and let $\tilde{\rm}$ be its associated  rooted planar bipartite map. Write $\Delta$ for the maximal vertex degree of $\rm$. Note that $\Delta = \Delta_\circ(\tilde \rm)$. There is a natural one-to-one correspondence between quasi-Eulerian orientations of $\rm$ and $\alpha_{\Delta}$-orientations of $\tilde{\rm}$. Precisely, given a quasi-Eulerian orientation $\cO$ of $\rm$, the associated $\alpha_{\Delta}$-orientation $\tilde{\cO}$ of $\tilde{\rm}$ is defined as follows. For any edge $\tilde{e}=\{h_\circ,h_\bullet\}\in \mathrm{E}(\tilde{\rm})$:
\begin{equation*}
    \begin{cases*}
      \tilde{\cO}(h_\circ) \coloneqq \cO(h_\circ), \\
      \tilde{\cO}(h_\bullet) \coloneqq (\Delta+1) - \cO(h_\circ).
    \end{cases*}
\end{equation*}
Accessibility and minimality are preserved. Thus, as a consequence of Proposition~\ref{prop: Correspondence alpha-orientation and geodesic labeling}, we recover as a special case, the correspondence  between the minimal quasi-Eulerian orientation of a plane map and the geodesic labeling of its dual map.

\section{Recovering mobiles and BDG bijection with \texorpdfstring{$\alpha_d$}{alpha-d}-orientations}\label{AppB}

In this section, we apply to the class of $\alpha_d$-orientations, the generic bijective scheme introduced by Bernardi and Fusy in \cite{BernardiFusy_TrigQuadPent}, and further expanded in \cite{BernardiFusy_Girth, BernardiFusy_Boundaries, BernardiFusy_Hypermaps}.
This allows us to recover the well-known bijection between Eulerian planar maps and  ``mobiles'', due to Bouttier, Di Francesco and Guitter \cite{BouttierDiFrancescoGuitter_Mobiles}. As already mentioned in the introduction, the ``mobile'' bijection was already recovered as a special case of a generic framework developed for hypermaps by Bernardi and Fusy~\cite{BernardiFusy_Hypermaps}. The novelty of our result is that we apply here their ``classical'' framework, i.e. the one they developed for classical maps in~\cite{BernardiFusy_TrigQuadPent} and not for hypermaps.

In Bernardi--Fusy's bijection and in Bouttier--Di Francesco--Guitter's bijection, the central objects are decorated trees that are referred to as \emph{mobiles}.  However, the definition of mobiles differ between the two articles \cite{BernardiFusy_TrigQuadPent, BouttierDiFrancescoGuitter_Mobiles}. To avoid confusion, we will refer to the mobiles in Bernardi--Fusy’s bijection as \emph{blossoming mobiles} (since they carry opening and closing stems), and the mobiles in Bouttier--Di Francesco--Guitter’s bijection as \emph{labeled mobiles} (since they are labeled!).\\

In contrast to the rest of the article, this appendix concerns exclusively \emph{unrooted} bipartite plane maps.

\subsection{Bernardi-Fusy's bijective scheme.}

Following \cite{BernardiFusy_TrigQuadPent}, a \emph{blossoming mobile} is an unrooted plane blossoming tree $\rt$ endowed with an orientation $\cO$, and which satisfies the following properties: 
\begin{itemize}
	\item the vertices of the tree can be of 2 types: either round or square, 
	\item the dangling half-edges are opening stems (i.e. outgoing) and are incident to square vertices, and
	\item for any half-edge $h\in \mathrm{H}(\rt)$, $\cO(h)=0$ if and only if $h$ is an outgoing stem or if $h$ is part of a round-square edge and is incident to the round vertex. Otherwise, we require that $\cO(h)>0$.
\end{itemize}

The \emph{excess} of a blossoming mobile is the number of its half-edges adjacent to a round vertex minus the number of its opening stems. Note that it differs from the notion of excess for blossoming trees  defined in Section~\ref{sec: Planted trees}. Finally, the \emph{weight} of an edge in a blossoming mobile is the sum of the orientation of its two half-edges.\medskip

The generic bijective scheme of~\cite{BernardiFusy_TrigQuadPent} is a bijective operation $\Phi_{\mathrm{BF}}$ that associates to a blossoming map $\rm$ endowed with an accessible and minimal orientation $\cO$, a blossoming mobile. 

\begin{rema}
Note that the maps considered here are unrooted, so the notion of accessibility needs to be adapted. In this appendix, following~\cite{BernardiFusy_TrigQuadPent} we say that an orientation on a plane map is \emph{accessible} if there exists an outer vertex that is accessible from every other vertex.
\end{rema}

Let us recall the definition of $\Phi_{\mathrm{BF}}(m,\cO)$: First, insert a square vertex into each face of $\rm$, and then apply the following \emph{local transformation} to each edge of $\rm$, see Figure~\ref{fig: BF local bijection}:

\begin{itemize}
	\item The \emph{local transformation} of a non-saturated edge  $e\in \mathrm{E}(\rm)$  consists in keeping it and adding an opening stem, pointing towards $e$, on each of the two square vertices placed in its two adjacent faces. The orientation of $e$ is preserved.

	\item The \emph{local transformation} of a saturated edge  $e\in \mathrm{E}(\rm)$ from a vertex $u$ to a vertex $v$, consists in adding an opening stem, pointing towards $e$, on the square vertex in its left face, then creating an edge between the square vertex in its right face and $v$, and removing $e$. The created square-round edge inherits the orientation of $e$, meaning that it is saturated from the square vertex to $v$.
\end{itemize}

Then, one of the main results of~\cite{BernardiFusy_TrigQuadPent} is:
\begin{figure}
	\centering
	\includegraphics[width=0.4\linewidth ,page=2]{BFBijection.pdf}\qquad\qquad
	\includegraphics[width=0.4\linewidth ,page=1]{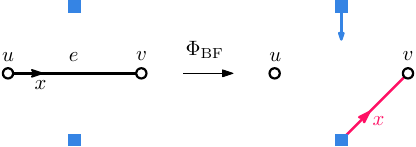}
	\caption{\label{fig: BF local bijection} The local transformation of an edge following Bernardi-Fusy's bijection: the non-saturated case (left), and the saturated case (right).}
\end{figure}

\begin{theo}[\cite{BernardiFusy_TrigQuadPent}, Theorem 11]\label{theo: BernardiFusy - Bijection}
The mapping $\Phi_{\mathrm{BF}}$ is a bijection between the set of unrooted plane maps endowed with a minimal accessible orientation and the set of blossoming mobiles of positive excess. 
\end{theo}

\begin{rema}
The bijective mapping $\Phi_{\mathrm{BF}}$ induces several parameter correspondences. Specifically, let  $\rm$ be a plane map endowed with a minimal accessible orientation $\cO$, and let $\rt=\Phi_{\mathrm{BF}}(\rm,\cO)$. There are  correspondences between : 
\begin{itemize}
	\item the degrees of the outer face of $\rm$ and the excess of $\rt$,
	\item the degrees of the other faces of $\rm$ and the degrees of the square vertices of $\rt$,
	\item the indegrees of the vertices of $\rm$ and the indegrees of the round vertices of $\rt$,
	\item the saturated edges of $\rm$ and the round-square edges in $\rt$.
\end{itemize}
Moreover, if $\cO$ is $k$-fractional, then every edge of $\rt$ has weight $k$.
\end{rema}

\begin{rema}
The previous theorem is a specific case of the more general bijective scheme defined by Bernardi and Fusy in \cite{BernardiFusy_TrigQuadPent}, tailored to fit our presentation. For the sake of simplicity and consistency with the rest of the present article, we have also slightly adapted their conventions for defining the transformation of maps.
\end{rema}

\subsection{Application to \texorpdfstring{$\alpha_d$}{alpha}-orientations}\label{sub:BF_application_to_alpha_d_} To apply Theorem \ref{theo: BernardiFusy - Bijection} to the context of bipartite maps, we need to introduce the following family of blossoming mobiles.
Fix $d\geq 1$. A \emph{$d$-blossoming mobile} is a blossoming mobile $\rt$, where round vertices are additionally colored in black and white and which additionally satisfies the following: 
\begin{enumerate}[(1)]
	\item it has positive excess,
	\item all edges are either between a white and a black vertices or between a white vertex and a square vertex,
	\item with maximum white vertex degree $d$,
\end{enumerate}
and endowed with a $(d+1)$-fractional orientation $\cO : \mathrm{H}(\rt) \rightarrow \mathbb{Z}_{\geq 0}$ such that
\begin{enumerate}[resume]
	\item every white vertex $v$ satisfies $\ind(v)=d\cdot\deg(v)$,
	\item every black vertex $v$ verifies $\ind(v)=\mathrm{d}(v)$, where  $\mathrm{d}(v)$ is the degree of $v$ in the oriented map ${\Phi_{\mathrm{BF}}}^{-1}(\rt)$, and
	\item without saturated white-black edges.
\end{enumerate}

Then Theorem~\ref{theo: BernardiFusy - Bijection} can be specialized to bipartite maps endowed with their minimal $\alpha_d$-orientation, as follows, see Figure \ref{fig: Link between BF & BDG} for an example. 

\begin{coro}\label{cor: BF bij on alpha-orientations}
Fix $d\geq1$. The mapping $\Phi_{\mathrm{BF}}$ induces a bijection between the set of bipartite plane maps with maximum white vertex degree $d$ and endowed with their minimal $\alpha_d$-orientation, and the set of $d$-blossoming mobiles. 
\end{coro}
\begin{proof}
Let $\rm$ be a bipartite plane map, with maximum white vertex degree $d$, and endowed with its minimal $\alpha_d$-orientation $\cO$, and let $\rt=\Phi_{\mathrm{BF}}(\rm,\cO)$. We start by proving that $\rt$ is a $d$-blossoming mobile.

First, as a direct consequence of Theorem~\ref{theo: BernardiFusy - Bijection}, $\rt$ has positive excess so that (1) holds. 

Then, every edge $e$ of $\rt$ either corresponds to the local transformation of a non-saturated edge in $(\rm,\cO)$, in which case it connects a black vertex and a white vertex, or it corresponds to a saturated edge $\tilde{e}$ in $(\rm,\cO)$. In the latter case, Property~\ref{claim: Saturated edges are from black to white endpoints} ensures that $\tilde{e}$ is saturated from its black endpoint to its white endpoint. Then, the local transformation imply that $e$ is a white-square edge, see Figure~\ref{fig: BF local bijection}. Properties (2) and (6) then follow.  

Finally, it is straightforward that $\ind(v)=d\cdot\deg_\rm(v)$ for any $v\in V_\circ(\rm)$, and $\ind(v)=\deg_\rm(v)$ for any $v\in V_\bullet(\rm)$, so that Properties (4) and (5) follow.\\

Reciprocally, let $\rt$ be a $d$-blossoming tree. It has positive excess by (1), allowing us to define $(\rm,\cO)={\Phi_{\mathrm{BF}}}^{-1}(\rt)$, as stated in Theorem~\ref{theo: BernardiFusy - Bijection}.

First, we prove that $\rm$ is bipartite. For any ${\oslash},\otimes\in\{\bullet,\circ,\text{\tiny{$\square$}}\}$, and $\mathrm{x}\in\{\rt,\rm\}$, denote by $n_{\oslash\otimes}(\mathrm{x})$ the number of edges in $\mathrm{x}$ between two vertices of type $\oslash$ and $\otimes$. Similarly, we denote respectively by $n^{\mathrm{sat}}_{\oslash\otimes}(\mathrm{x})$ and $n^{\mathrm{non sat}}_{\oslash\otimes}(\mathrm{x})$ the number of saturated and non-saturated edges of this type. Finally, let $\rV_\oslash(\mathrm{x})$ denote the set of vertices of $\mathrm{x}$ of type $\oslash$. With these notations in place, we aim to demonstrate that $n_{\bullet\bullet}(\rm)=n_{\circ\circ}(\rm)=0$.

By Property~(2), every vertex adjacent to a square vertex in $\rt$ is both round and white, so that the edges of $\rm$ consist of the round-round edges of $\rt$ (which are also black-white) together with additional saturated edges connecting any colored vertex to a white vertex. Hence, one has:
\begin{equation*}
	n_{\bullet\bullet}(\rm)=0
	\quad \text{and} \quad 
	n_{\circ\circ}(\rm)=n^{\mathrm{sat}}_{\circ\circ}(\rm),
\end{equation*}
and from the correspondences induced by $\Phi_{\mathrm{BF}}$ and (6), one has:
\begin{equation}\label{eq: Enumeration in proof of BF}
	n_{\circ\sqempt}(\rt) = n^{\mathrm{sat}}_{\bullet\circ}(\rm) + n^{\mathrm{sat}}_{\circ\circ}(\rm)  \quad \text{and} \quad  n_{\bullet\circ}(\rt) = n^{\mathrm{nonsat}}_{\bullet\circ}(\rt) =  n^{\mathrm{nonsat}}_{\bullet\circ}(\rm) .
\end{equation}
Recall that in a blossoming mobile the orientation of every half-edge belonging to a round-square edge and being incident to the round vertex is zero, and that $n_{\bullet\sqempt}(\rt)=n_{\bullet\bullet}(\rt)=n_{\circ\circ}(\rt)=0$. Therefore, it follows that:
\begin{equation}\label{eq:degout_deginn}
	\sum\limits_{v\in \rV_\circ(\rt)}{\out(v)} =\sum\limits_{v\in \rV_\bullet(\rt)}{\ind(v)}.
\end{equation}
Moreover, by Property~(4) and~\eqref{eq: Enumeration in proof of BF}, on the one hand, we have: 
\begin{equation*}
	\sum\limits_{v\in \rV_\circ(\rt)}{\out(v)} = \sum\limits_{v\in \rV_\circ(\rt)}{\deg(v)} = n_{\circ\sqempt}(\rt) + n_{\bullet\circ}(\rt) =  n^{\mathrm{sat}}_{\bullet\circ}(\rm) +  n^{\mathrm{nonsat}}_{\bullet\circ}(\rm) + n^{\mathrm{sat}}_{\circ\circ}(\rm).
\end{equation*}
On the other hand, by~Property (5), we have:
\begin{equation*}
	\sum\limits_{v\in \rV_\bullet(\rt)}{\ind(v)} = \sum\limits_{v\in \rV_\bullet(\rm)}{\deg(v)} = n^{\mathrm{sat}}_{\bullet\circ}(\rm) + n^{\mathrm{nonsat}}_{\bullet\circ}(\rm).
\end{equation*}
We conclude that $n^{\mathrm{sat}}_{\circ\circ}(\rm)=0$. Thus, $\rm$ is bipartite. 

Since the white vertex degree is preserved by $\Phi_{\mathrm{BF}}$, Property (3) ensures that $\rm$ has maximum white vertex degree $d$. 
Lastly, $\cO$ is an $\alpha_d$-orientation on $\rm$, which follows from the definition of $\Phi_{\mathrm{BF}}$ and from Properties (4) and (5). 
\end{proof}

\subsection{Recovering the Bouttier-Di Francesco-Guitter bijection for Eulerian maps} In this section, we show that the bijection presented in the previous section is, in fact, equivalent -- up to minor adjustments of the trees -- to the celebrated BDG bijection, introduced by Bouttier, Di Francesco, and Guitter in~\cite{BouttierDiFrancescoGuitter_Mobiles}. This is illustrated on Figure~\ref{fig: Link between BF & BDG}. 

\subsubsection{BDG bijection}
Let us first recall the BDG bijection. We begin by introducing several definitions.
A \emph{labeled mobile} is a plane tree, in which vertices can be of 3 types: either black and round, white and round or square, with the requirement that all edges are either of type white-black or white-square. Additionally,  labels are assigned to square vertices and to each side of any white-black edge -- the latter are referred to as \emph{flags labels} -- subject to the following conditions:
\begin{enumerate}[(I)]
	\item each flag label is non-negative, and at least one is zero,
	\item each square vertex has a positive label,
	\item in clockwise order around a black vertex, two consecutive labels $\ell_1$ and $\ell_2$ verify  $\ell_2\leq \ell_1$ if they are flag labels on the same edge, and $\ell_2\geq \ell_1$ otherwise,
	\item in clockwise order around a white vertex, two consecutive labels $\ell_1$ and $\ell_2$ verify $\ell_2\geq \ell_1$ if they are flag labels on the same edge, $\ell_2= \ell_1-1$ if $\ell_1$ is a square vertex label, and $\ell_2=\ell_1$ otherwise.
\end{enumerate}

\begin{figure}
	\centering
	\includegraphics[width=0.4\linewidth ,page=2]{BDGBijection.pdf}\qquad\qquad
	\includegraphics[width=0.4\linewidth ,page=1]{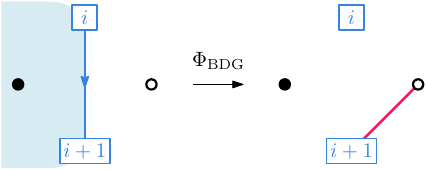}
	\caption{\label{fig: BDG local bijection} The local transformation of an edge following Bouttier-Di Francesco-Guitter's bijection: the non-geodesic case (left), and the geodesic case (right).}
\end{figure}

For sake of simplicity, we assume that each corner incident to a square vertex inherits the label of that vertex.  The \emph{successor of a corner} labeled $\ell\geq 2$ is defined as the first corner labeled  $\ell - 1$ encountered when  moving clockwise around the tree. The \emph{successor of a flag} $\ell\geq 1$ is defined as the first corner labeled  $\ell$ encountered when  moving clockwise around the tree. 
Conversely, the \emph{predecessors} of a corner are the corners and flags for which it is the successor.

Recall the definition and properties of Eulerian maps from Section~\ref{sub:duality}. In particular, every Eulerian map endowed with one of its proper 2-colorings of its faces, possesses a canonical orientation of its edges by requiring that its white faces lie on the left of any directed edge. 
Given a pointed bicolored Eulerian planar map, we label its vertices by their geodesic labeling, see Section~\ref{sub:directed_geodesic_labeling}. 
Then, we define the transformation $\Phi_{\mathrm{BDG}}$ by applying to every edge, the local transformation illustrated in Figure \ref{fig: BDG local bijection}. 

One of the main results of~\cite{BouttierDiFrancescoGuitter_Mobiles} is the following: 
\begin{theo}[\cite{BouttierDiFrancescoGuitter_Mobiles}, Section 3.1]\label{theo: BDG bijection}
The application $\Phi_\mathrm{BDG}$ is a bijection between the set of pointed Eulerian planar maps and the set of labeled mobiles.
\end{theo}

\begin{rema}
The bijective mapping $\Phi_\mathrm{BDG}$ induces several parameter correspondences. Let $\rm$ be a pointed Eulerian planar map, and write $\rt=\Phi_\mathrm{BDG}(\rm)$. Then, we have the following correspondences between:
\begin{itemize}
	\item the degree of the pointed vertex of $\rm$ and the number of corners with label $1$ plus the number of flags with label $0$ in $\rt$,
	\item the degree of the white faces of $\rm$ and the white (round) vertex degrees in $\rt$,
	\item the directed geodesic labels of $\rm$ and the labels of the square vertices in $\rt$.
\end{itemize}
\end{rema}

\begin{rema}\label{rema: BDG bijection}
Note that if the white vertex degree of $\rt$ is at most $d\geq 1$, then $\rm$ has white face degree at most $d$. Hence, the directed geodesic labeling between two adjacent vertices of $\rm$ is at most $d-1$, since the contour of their adjacent white face forms a directed path connecting them.

It will be important for the following to observe that, this implies that the labels of both flags on an edge of $\rt$ can differ by at most $d-1$.
\end{rema}
This remark justifies the following definition. A \emph{$d$-labeled mobile} is a labeled mobile with maximum white vertex degree $d$. 

\subsubsection{$d$-labeled mobiles are $d$-blossoming mobiles}
In this section, we prove that for every bipartite plane map $\rm$ (with its minimal $\alpha_d$-orientation $\cO)$, the blossoming mobile $\rt:=\Phi_{\mathrm{BF}}(\rm,\cO)$ is close to the labeled mobile $\tilde \rt:=\Phi_\mathrm{BDG}(\rm^\dagger)$. Specifically, the  vertices and edges of $\rt$ and $\tilde \rt$ coincide, and the labels of $\tilde \rt$ are replaced by some orientation on $\rt$.

More precisely, for $d\geq 1$, we define an operation $\Upsilon_d$ on $d$-labeled mobiles as follows. Fix $\rt$ a $d$-labeled mobile. First, add to each corner of $\rt$ incident to a square vertex as many opening stems as it possesses predecessors. Then, endow the resulting tree with the orientation $\cO:\mathrm{H}(t)\rightarrow \Z_{\geq 0}$ defined by:
\begin{itemize}
	\item For every white-square edge $e=\left\{h_\circ, h_\sqempt\right\}$, set $\cO(h_\circ)\coloneqq 0$ and $\cO(h_\sqempt)\coloneqq d+1$. 
	\item For every white-black edge $e=\left\{h_\circ, h_\bullet\right\}$ with flags $\ell_1$ and $\ell_2$, in the clockwise order around the white vertex, set $\cO(h_\circ)\coloneqq \ell_2-\ell_1+1$ and $\cO(h_\bullet)\coloneqq d+1-\cO(h_\circ)$. Observe that, as  a directly  consequence of Condition (III) in the definition of labeled mobiles and of Remark~\ref{rema: BDG bijection}, one has $\cO(h_\circ)\geq 1$ and $\cO(h_\bullet)\geq 0$.
\end{itemize} 
Finally, forget about the flags and the vertex labeling. The resulting tree is denoted by $\Upsilon_d(\rt)$. Then, we have:

\begin{lemm}\label{lemm:labeled and blossoming mobile}
Fix $d\geq 1$. For any $d$-labeled mobile $\rt$, $\Upsilon_d(\rt)$ is a $d$-blossoming mobile.
\end{lemm}
This leads us to the following result, illustrated in Figure~\ref{fig: Link between BF & BDG}:
\begin{prop}\label{prop:recover BDG}
Fix $d\geq 1$. On the set of pointed Eulerian planar maps with maximum white vertex degree $d$, we have: 
\begin{equation}\label{eq: commutation de BF/BDG}		
	\Phi_\mathrm{BF}\circ{\mathrm{Dual}}=\Upsilon_d\circ\Phi_\mathrm{BDG}.
\end{equation}
\end{prop}

\begin{figure}
  	\centering
  	\includegraphics[width=1\linewidth ,page=7]{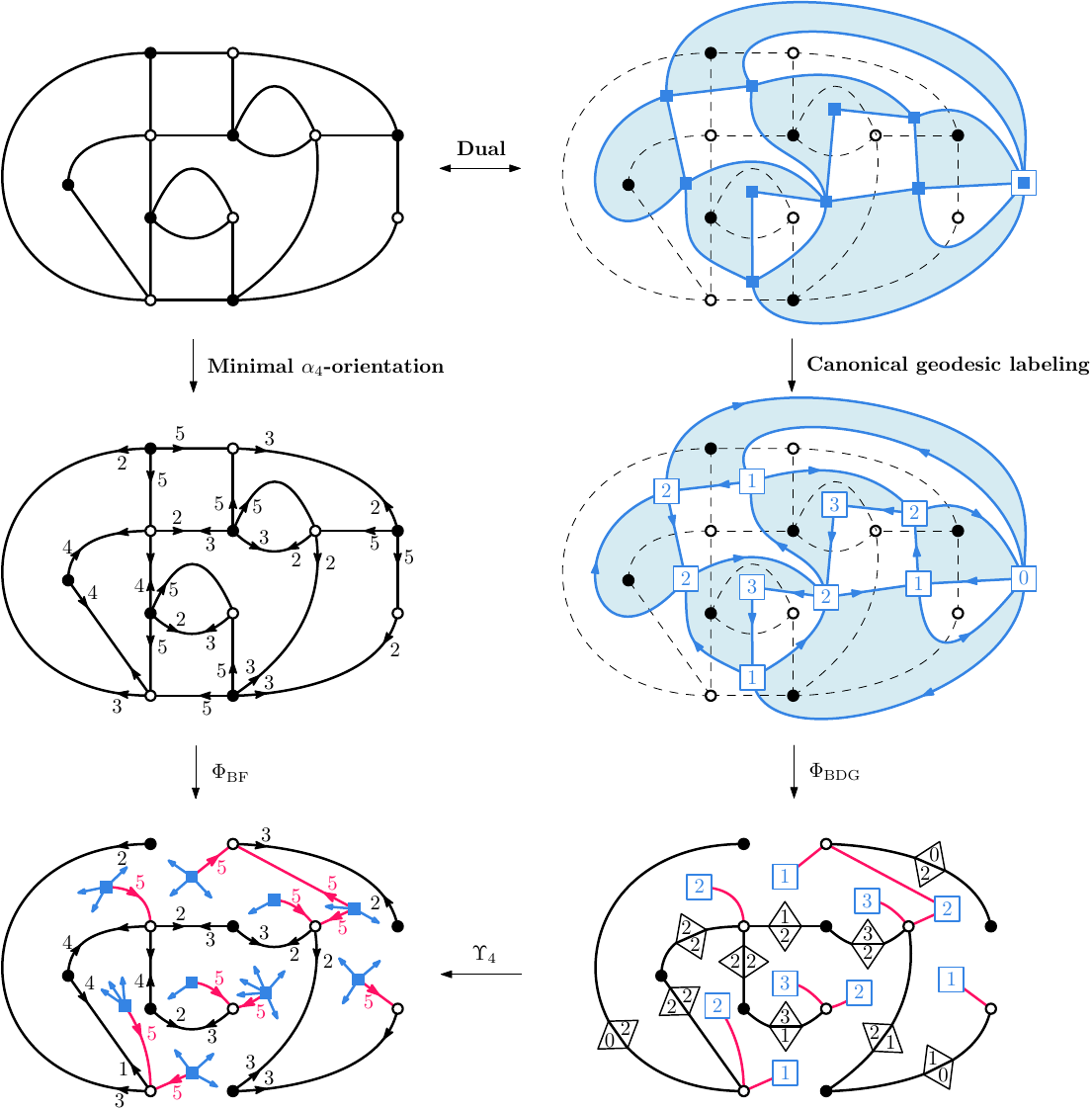}
 	\caption{\label{fig: Link between BF & BDG}A bipartite plane map and its successive transformations illustrating that Bernardi-Fusy's bijective scheme, when applied to $\alpha_d$-orientations, coincides with Bouttier-Di Francesco-Guitter's bijection.}
\end{figure}

We begin by proving the proposition, relying on Lemma~\ref{lemm:labeled and blossoming mobile}. 
\begin{proof}
We first show that $\Upsilon_d$ is an injective mapping from the set of $d$-labeled mobiles into the set of $d$-blossoming mobiles. Let $\rt$ be a $d$-labeled mobile, and denote $\tilde \rt = \Upsilon_d(\rt)$. Since both mobiles share the same underlying undecorated tree, reconstructing $\rt$ from $\tilde{\rt}$ reduces to recovering the value of its labels.

The key idea is that it suffices to reconstruct the successor function on $\rt$ from $\tilde \rt$.
 Indeed, corners with no successor have label $1$, and since the successor of a corner labeled $l$ is a corner labeled $l-1$, we can recover the other corner labels iteratively. Similarly, flags with no successor have label $0$, and otherwise they have the same label as their successors, which are corners by definition. 

Now, from the opening stems of $\tilde{\rt}$, we can reconstruct the predecessor function: the counterclockwise cyclic contour of $\tilde{\rt}$ induces a matching between the opening stems and the set formed by the sides of the round–round edges together with the square corners that precede a round–square edge in the clockwise direction. Then, we can recover the successor function from the predecessor function. Therefore, we conclude that $\Upsilon_d$ is injective. 

Secondly, for any $n\geq 1$, the map $\Phi_\mathrm{BF}\circ\mathrm{Dual}\circ{\Phi_\mathrm{BDG}}^{-1}$ induces a bijection between the (finite) set of $d$-labeled mobiles with $n$ vertices and the (finite) set of $d$-blossoming mobiles with $n$ vertices. Hence, $\Upsilon_d$ defines a bijection between these two sets. 

To conclude, it remains to check that the restriction of $\Upsilon_d$ to the (finite) set of $d$-labeled mobiles with $n$ vertices coincides with $\Phi_\mathrm{BF}\circ\mathrm{Dual}\circ{\Phi_\mathrm{BDG}}^{-1}$. This follows from the fact that the closure operations ${\Phi_\mathrm{BF}}^{-1}$ and ${\Phi_\mathrm{BDG}}^{-1}$ both depend only on the successor function (see \cite[Section 3]{BernardiFusy_Girth} and \cite[Section 3.2]{BouttierDiFrancescoGuitter_Mobiles}, respectively), and that we proved that these functions coincide for any labeled mobile $\rt$ and its associated blossoming mobile $\Upsilon(\rt)$.
\end{proof}

\begin{proof}[Proof of Lemma~\ref{lemm:labeled and blossoming mobile}]
In this proof, Roman numerals refer to the conditions in the definition of labeled mobiles, while the Arabic numerals to those in the definition of $d$-blossoming mobiles given at the beginning of Section~\ref{sub:BF_application_to_alpha_d_}. Let $\rt$ be a $d$-labeled mobile, and let $\tilde{\rt}\coloneqq\Upsilon_d(\rt)$. We prove that $\tilde \rt$ verifies the conditions (1) to (6).

First, recall that the excess of a blossoming mobile is defined as the difference between the number of half-edges incident to round vertices and the number of opening stems. Next, the number of half-edges incident to round vertices in $\tilde \rt$ is the same as in $\rt$, and the latter is exactly equal to the number of its flags plus the number of its corners incident to square vertices. Therefore, the excess of $\tilde \rt$ is equal to the number of flags with label 0 (which are exactly the flags with no successor in the mobile by (I) and (II)) plus the number of corner labeled $1$ (which are exactly the labeled corner with no successor in the mobile by (II)) in $\rt$. Since there exists at least one flag labeled $0$ by  (I), Condition (1) holds. 

Conditions (2) and (3) are direct consequences of the definition of $\Upsilon_d$, which preserves the shape of $\rt$ and the type of its vertices. 

Next, let $v$ be a white vertex of $\rt$. Let $m$ and $n$ be the number of white-square and white-black edges incident to $v$, respectively. List the white-black edges as $e_1, \ldots, e_n$, in clockwise order around $v$, starting arbitrarily from any one of them. For any $1\leq i \leq n$, denote by $\ell^\text{left}_i, \ell^\text{right}_i$ the labels of the flags situated on the left and on the right of $e_i$. 

Since both mobiles $\rt$ and $\tilde \rt$ share the same underlying undecorated tree, we will also refer to $v$ and $e_1, \ldots, e_n$ for their corresponding vertex and edges in $\tilde \rt$. 
With these notations, the orientation $\cO$ of $\tilde\rt$ satisfies the following:
\begin{equation*}
    \cO(h_i) = \ell^\text{right}_i - \ell^\text{left}_i + 1,
\end{equation*}
for every $1\leq i\leq n$, and where $h_i$ is the half-edge of $e_i$ incident to $v$. Moreover, note that $\ell^\text{right}_i$ and $\ell^\text{left}_{i+1}$ are consecutive labels of flags around $v$ in the clockwise order in $\rt$. Then, it follows from (IV) that $\ell^\text{right}_i-\ell^\text{left}_{i+1}$ is the number of white-square edges between $e_i$ and $e_{i+1}$ in $\rt$. Thus, it is also the case for $\tilde{\rt}$. Therefore, by rewriting the following sum, we obtain:
\begin{equation*}
\out(v) = \sum\limits_{1\leq i\leq n}{(\ell^\text{right}_i-\ell^\text{left}_i+1)} = n + \sum\limits_{1\leq i\leq n}{(\ell^\text{right}_i-\ell^\text{left}_{i+1})} =  n+m = \deg(v)
\end{equation*}
where $\ell^\text{left}_{n+1}\coloneqq \ell^\text{left}_{1}$. Finally, as the orientation $\cO$ is a $(d+1)$-fractional orientation, one has ${\ind(v)=d\cdot\deg(v)}$. Hence, (4) follows.

Due to comparable considerations for black vertices, (III) implies (5).
Finally, as shown at the beginning of this proof, $\cO(h_\circ)>0$ for any white-black edge $e=\left\{h_\circ, h_\bullet\right\}$. In particular, this implies (6).
\end{proof}

\section*{Table of notations}\label{sec:notations}
\subsection*{Planar maps\nopunct}\leavevmode\par
\begin{tabular}{p{2.5cm}p{11.5cm}}
$\Delta(\rm)$ & maximal vertex degree of a bipartite map $\rm$,\\
$\Delta_\bullet(\rm)$, $\Delta_\circ(\rm)$ & maximal vertex degree of black vertices and white vertices of a bipartite map $\rm$, respectively,\\
$\rV_\bullet(\rm)$, $\rV_\circ(\rm)$ & set of the (round) black vertices and (round) white vertices of a bipartite map $\rm$, respectively,\\
$\mathcal{M}$ &  set of rooted bipartite planar maps, \\
$\bar{\mathcal{M}}$ &  set of rooted bipartite plane maps, \\
${\mathcal{M}}^{(d)}$ &  set of rooted bipartite planar maps with maximal vertex degree $d$,\\
$\bar{\mathcal{M}}^{(d)}$ &  set of rooted bipartite plane maps with maximal vertex degree $d$,\\
$\mathcal{M}^\sq_\circ$ &  set of bipartite planar maps whose vertices of degree two are either round or square and the other vertices are all round, and rooted at a white vertex,\\
$\mathcal{I}_\circ$ &  set of rooted planar maps endowed with a spin configuration, rooted at a white vertex.
\end{tabular}

\subsection*{Trees\nopunct}\leavevmode\par
\begin{tabular}{p{2.5cm}p{11.5cm}}$c(\rt)$ &  charge of a blossoming tree $\rt$,\\
$c_\rt(v)$ or $c(v)$ &  charge of a vertex $v$ of a blossoming tree $\rt$,\\
$\mathrm{exc}(\rt)$ &  excess of a planted $\alpha_d$-tree $\rt$,\\
$\mathcal{T}_k$ &  set of well-charged trees of total charge $k$,\\
${\mathcal{T}}^{(d)}_k$ &  set of well-charged trees of total charge $k$ and maximal vertex degree $d$.
\end{tabular}

\subsection*{Weights\nopunct}\leavevmode\par
\begin{tabular}{p{2.5cm}p{11.5cm}}
$w^\mathrm{tree}(\rt)$ &  weight of well-charged planted trees $\rt$,\\
$w(\rm)$ &  weight of a map $\rm\in\mathcal{M}$,\\
$\bar{w}(\rm)$ &  weight of a map $\rm\in\bar{\mathcal{M}}$,\\
$w^\sq(\rm)$ &  weight of a map $\rm\in\mathcal{M}^\sq_\circ$,\\
$w^\mathrm{Ising}(\rm,\sigma)$ &  weight of a map $\rm$ endowed with a spin configuration $\sigma$,
\end{tabular}

\subsection*{Generating functions\nopunct}\leavevmode\par
\begin{tabular}{p{2.5cm}p{11.5cm}}
$W(u,\xi)$ & weighed generating series of white-rooted well-charged planted trees,\\
$B(u,\xi)$ & weighed generating series of black-rooted well-charged planted trees,\\
$B_1(\ux,\uy, u)$ & weighed generating series of black-rooted well-charged planted trees with charge 1,\\
${M}_\circ(u)$ & weighed generating series of maps $\rm\in{\mathcal{M}}_\circ$,\\
$\bar{M}_\circ(u)$ & weighed generating series of maps $\rm\in\bar{\mathcal{M}}_\circ$,\\
${M}^\sq_\circ(u)$ & weighed generating series of maps $\rm\in{\mathcal{M}}^\sq_\circ$,\\
${I}_\circ(\ux,\uy, t,\nu,u)$ & weighed generating series of maps $\rm\in{\mathcal{I}}_\circ$,\\
$\Theta$ & function on $\Q\llbracket\ux,\uy, t, \nu, u\rrbracket$, interpreted as a change of variables.
\end{tabular}


\bibliographystyle{plainnat} 
\bibliography{main}  

@misc{Maple,
	author = {Albenque, M. and M\'enard, L. and Tokka, N.},
	title={Maple companion file},
	howpublished={Available on the authors' webpage:  \url{https://www.irif.fr/~malbenque/}, \url{https://www.normalesup.org/~menard/} and \url{https://www.irif.fr/~tokka/}},
}

@PHDTHESIS{SchaefferPhd,
	title = "Conjugaison d'arbres et cartes combinatoires aléatoires",
	author = "Schaeffer, G.",
	year = "1998",
	pages = "229 p",
	school = "Univ. Bordeaux 1",
}

@article{Lepoutre19,
    AUTHOR = {Lepoutre, M.},
     TITLE = {Blossoming bijection for higher-genus maps},
   JOURNAL = {J. Combin. Theory Ser. A},
  FJOURNAL = {Journal of Combinatorial Theory. Series A},
    VOLUME = {165},
      YEAR = {2019},
     PAGES = {187--224},
}

@misc{AlbenqueMenard,
      title={Geometric properties of spin clusters in random triangulations coupled with an {I}sing Model}, 
      author={Albenque, M. and Ménard, L.},
      year={2022},
      archivePrefix={arXiv},
      eprint={2201.11922},
      primaryClass={math.PR},
      note = "arXiv:2201.11922",
}

@misc{Turunen,
      TITLE={Interfaces in the vertex-decorated {I}sing model on random triangulations of the disk}, 
      author={Turunen, J.},
      year={2020},
      eprint={2003.11012},
      archivePrefix={arXiv},
      primaryClass={math-ph},
      note = "arXiv:2003.11012",
}

@article{ChenTurunen22,
   title={{I}sing Model on Random Triangulations of the Disk: Phase Transition},
   volume={397},
   number={2},
   journal={Communications in Mathematical Physics},
   publisher={Springer Science and Business Media LLC},
   author={Chen, L. and Turunen, J.},
   year={2022},
   month=dec, pages={793–873},
   }

@article{ChenTurunen20,
   title={Critical {I}sing Model on Random Triangulations of the Disk: Enumeration and Local Limits},
   volume={374},
   number={3},
   journal={Communications in Mathematical Physics},
   publisher={Springer Science and Business Media LLC},
   author={Chen, L. and Turunen, J.},
   year={2020},
   month=jan, pages={1577–1643}
   }

@article {IsingAMS,
     AUTHOR = {Albenque, M. and M\'enard, L. and Schaeffer, G.},
     TITLE = {Local convergence of large random triangulations coupled with
              an {I}sing model},
   JOURNAL = {Trans. Amer. Math. Soc.},
  FJOURNAL = {Transactions of the American Mathematical Society},
    VOLUME = {374},
      YEAR = {2021},
    NUMBER = {1},
     PAGES = {175--217},
}

@Article{Itzykson1980,
  author   = {Itzykson, C. and Zuber, J. B.},
  journal  = {J. Math. Phys.},
  title    = {The planar approximation. {II}},
  year     = {1980},
  number   = {3},
  pages    = {411--421},
  volume   = {21},
  fjournal = {Journal of Mathematical Physics},
}

@article {Kazakov86,
  AUTHOR = {Kazakov, V. A.},
     TITLE = {Ising model on a dynamical planar random lattice: exact solution},
   JOURNAL = {Phys. Lett. A},
  FJOURNAL = {Physics Letters. A},
    VOLUME = {119},
      YEAR = {1986},
    NUMBER = {3},
     PAGES = {140--144},
}

@article {KazakovBoulatov87,
    AUTHOR = {Boulatov, D. V. and Kazakov, V. A.},
     TITLE = {The {I}sing model on a random planar lattice: the structure of
              the phase transition and the exact critical exponents},
   JOURNAL = {Phys. Lett. B},
  FJOURNAL = {Physics Letters. B. Particle Physics, Nuclear Physics and
              Cosmology},
    VOLUME = {186},
      YEAR = {1987},
    NUMBER = {3-4},
     PAGES = {379--384},
}

@Article{BeCa94,
  author     = {Bender, E. A. and Canfield, E. R.},
  journal    = {SIAM J. Discrete Math.},
  title      = {The number of degree-restricted rooted maps on the sphere},
  year       = {1994},
  number     = {1},
  pages      = {9--15},
  volume     = {7},
  fjournal   = {SIAM Journal on Discrete Mathematics},
}

@article {Tutte63,
    AUTHOR = {Tutte, W. T.},
     TITLE = {A census of planar maps},
   JOURNAL = {Canadian J. Math.},
  FJOURNAL = {Canadian Journal of Mathematics. Journal Canadien de
              Math\'ematiques},
    VOLUME = {15},
      YEAR = {1963},
     PAGES = {249--271},
}

@article {Tutte95,
    AUTHOR = {Tutte, W. T.},
     TITLE = {Chromatic sums revisited},
   JOURNAL = {Aequationes Math.},
  FJOURNAL = {Aequationes Mathematicae},
    VOLUME = {50},
      YEAR = {1995},
    NUMBER = {1-2},
     PAGES = {95--134},
}

@article{Tutte62, 
	TITLE={A Census of Slicings}, 
	VOLUME={14}, 
	JOURNAL={Canadian Journal of Mathematics}, 
	AUTHOR={Tutte, W. T.}, 
	YEAR={1962}, 
	PAGES={708–722},
}

@article {Tutte68,
    AUTHOR = {Tutte, W. T.},
     TITLE = {On the enumeration of planar maps},
   JOURNAL = {Bull. Amer. Math. Soc.},
  FJOURNAL = {Bulletin of the American Mathematical Society},
    VOLUME = {74},
      YEAR = {1968},
     PAGES = {64--74},
}

@article {Tutte64,
    AUTHOR = {Brown, W. G. and Tutte, W. T.},
     TITLE = {On the enumeration of rooted non-separable planar maps},
   JOURNAL = {Canadian J. Math.},
  FJOURNAL = {Canadian Journal of Mathematics. Journal Canadien de
              Math\'ematiques},
    VOLUME = {16},
      YEAR = {1964},
     PAGES = {572--577},
}

@article{THooft74,
title = {A planar diagram theory for strong interactions},
journal = {Nuclear Physics B},
volume = {72},
number = {3},
pages = {461-473},
year = {1974},
author = {Hooft, G.'t },
}

@misc{AlbenqueBouttier_slices,
      title={The slice decomposition of planar hypermaps}, 
      author={Albenque, M. and Bouttier, J. },
      year={2025},
      eprint={2509.06850},
      archivePrefix={arXiv},
      primaryClass={math.CO},
      note = "arXiv:2509.06850",
}

@book {FlajoletSedgewick09,
    AUTHOR = {Flajolet, P. and Sedgewick, R.},
     TITLE = {Analytic combinatorics},
 PUBLISHER = {Cambridge University Press, Cambridge},
      YEAR = {2009},
     PAGES = {xiv+810},
}

@article {CoriVauquelin,
    AUTHOR = {Cori, R. and Vauquelin, B.},
     TITLE = {Planar maps are well labeled trees},
   JOURNAL = {Canadian J. Math.},
  FJOURNAL = {Canadian Journal of Mathematics. Journal Canadien de
              Math\'ematiques},
    VOLUME = {33},
      YEAR = {1981},
    NUMBER = {5},
     PAGES = {1023--1042},
}

@book {Eynard16,
    AUTHOR = {Eynard, B.},
     TITLE = {Counting surfaces},
    SERIES = {Progress in Mathematical Physics},
    VOLUME = {70},
      NOTE = {CRM Aisenstadt chair lectures},
 PUBLISHER = {Birkh\"auser/Springer, [Cham]},
      YEAR = {2016},
}

@article {PouSch06,
    AUTHOR = {Poulalhon, D. and Schaeffer, G.},
     TITLE = {Optimal coding and sampling of triangulations},
   JOURNAL = {Algorithmica},
  FJOURNAL = {Algorithmica. An International Journal in Computer Science},
    VOLUME = {46},
      YEAR = {2006},
    NUMBER = {3-4},
     PAGES = {505--527},
}

@article {Felsner04,
    AUTHOR = {Felsner, S.},
     TITLE = {Lattice structures from planar graphs},
   JOURNAL = {Electron. J. Combin.},
  FJOURNAL = {Electronic Journal of Combinatorics},
    VOLUME = {11},
      YEAR = {2004},
    NUMBER = {1},
     PAGES = {R15},
}

@article {Scha97,
    AUTHOR = {Schaeffer, G.},
     TITLE = {Bijective census and random generation of {E}ulerian planar
              maps with prescribed vertex degrees},
   JOURNAL = {Electron. J. Combin.},
  FJOURNAL = {Electronic Journal of Combinatorics},
    VOLUME = {4},
      YEAR = {1997},
    NUMBER = {1},
     PAGES = {R20},
}

@article {BernardiFusy_TrigQuadPent,
    AUTHOR = {Bernardi, O. and Fusy, E.},
     TITLE = {A bijection for triangulations, quadrangulations,
              pentagulations, etc.},
   JOURNAL = {J. Combin. Theory Ser. A},
  FJOURNAL = {Journal of Combinatorial Theory. Series A},
    VOLUME = {119},
      YEAR = {2012},
    NUMBER = {1},
     PAGES = {218--244},
}

@article {BernardiFusy_Girth,
    AUTHOR = {Bernardi, O. and Fusy, E.},
     TITLE = {Unified bijections for maps with prescribed degrees and girth},
   JOURNAL = {J. Combin. Theory Ser. A},
  FJOURNAL = {Journal of Combinatorial Theory. Series A},
    VOLUME = {119},
      YEAR = {2012},
    NUMBER = {6},
     PAGES = {1351--1387},
}

@article {BousquetMelouJehanne,
    AUTHOR = {Bousquet-M\'elou, M. and Jehanne, A.},
     TITLE = {Polynomial equations with one catalytic variable, algebraic
              series and map enumeration},
   JOURNAL = {J. Combin. Theory Ser. B},
  FJOURNAL = {Journal of Combinatorial Theory. Series B},
    VOLUME = {96},
      YEAR = {2006},
    NUMBER = {5},
     PAGES = {623--672},
}

@book {LandoZvonkin,
    AUTHOR = {Lando, S. K. and Zvonkin, A. K.},
     TITLE = {Graphs on surfaces and their applications},
    SERIES = {Encyclopaedia of Mathematical Sciences},
    VOLUME = {141},
      NOTE = {With an appendix by Don B. Zagier,
              Low-Dimensional Topology, II},
 PUBLISHER = {Springer-Verlag, Berlin},
      YEAR = {2004},
     PAGES = {xvi+455},
}

@article {BrezinItzyksonParisiZuber,
    AUTHOR = {Br\'ezin, E. and Itzykson, C. and Parisi, G. and Zuber, J. B.},
     TITLE = {Planar diagrams},
   JOURNAL = {Comm. Math. Phys.},
  FJOURNAL = {Communications in Mathematical Physics},
    VOLUME = {59},
      YEAR = {1978},
    NUMBER = {1},
     PAGES = {35--51},
}

@article {BernardiBousquetMelou,
    AUTHOR = {Bernardi, O. and Bousquet-M\'elou, M.},
     TITLE = {Counting colored planar maps: algebraicity results},
   JOURNAL = {J. Combin. Theory Ser. B},
  FJOURNAL = {Journal of Combinatorial Theory. Series B},
    VOLUME = {101},
      YEAR = {2011},
    NUMBER = {5},
     PAGES = {315--377},
}

@incollection {BousquetMelou06,
    AUTHOR = {Bousquet-M\'elou, M.},
     TITLE = {Counting planar maps, coloured or uncoloured},
 BOOKTITLE = {Surveys in combinatorics 2011},
    SERIES = {London Math. Soc. Lecture Note Ser.},
    VOLUME = {392},
     PAGES = {1--49},
 PUBLISHER = {Cambridge Univ. Press, Cambridge},
      YEAR = {2011},
}

@misc{BousquetMelouSchaeffer_Bipartite,
      title={The degree distribution in bipartite planar maps: applications to the {I}sing model}, 
      author={Bousquet-Melou, M. and Schaeffer, G.},
      year={2003},
      eprint={math/0211070},
      archivePrefix={arXiv},
      primaryClass={math.CO},
      note = "arXiv:math/0211070",
}

@misc{BousquetMelouCarranceLouf_Ising,
      title={The {I}sing model on cubic maps: arbitrary genus}, 
      author={Bousquet-Melou, M.  and Carrance, A. and Louf, B.},
      year={2025},
      note="arXiv:2504.00768",
}

@article {BernardiFusy_Boundaries,
    AUTHOR = {Bernardi, O. and Fusy, E.},
     TITLE = {Bijections for planar maps with boundaries},
   JOURNAL = {J. Combin. Theory Ser. A},
  FJOURNAL = {Journal of Combinatorial Theory. Series A},
    VOLUME = {158},
      YEAR = {2018},
     PAGES = {176--227},
}

@article {BouttierDiFrancescoGuitter_Census,
    AUTHOR = {Bouttier, J. and Di Francesco, P. and Guitter, E.},
     TITLE = {Census of planar maps: from the one-matrix model solution to a
              combinatorial proof},
   JOURNAL = {Nuclear Phys. B},
  FJOURNAL = {Nuclear Physics. B. Theoretical, Phenomenological, and
              Experimental High Energy Physics. Quantum Field Theory and
              Statistical Systems},
    VOLUME = {645},
      YEAR = {2002},
    NUMBER = {3},
     PAGES = {477--499},
}

@article {BouttierDiFrancescoGuitter_Mobiles,
    AUTHOR = {Bouttier, J. and Di Francesco, P. and Guitter, E.},
     TITLE = {Planar maps as labeled mobiles},
   JOURNAL = {Electron. J. Combin.},
  FJOURNAL = {Electronic Journal of Combinatorics},
    VOLUME = {11},
      YEAR = {2004},
    NUMBER = {1},
     PAGES = {Research Paper 69, 27},
}

@article {AlbenquePoulalhon_Generic,
    AUTHOR = {Albenque, M. and Poulalhon, D.},
     TITLE = {A generic method for bijections between blossoming trees and
              planar maps},
   JOURNAL = {Electron. J. Combin.},
  FJOURNAL = {Electronic Journal of Combinatorics},
    VOLUME = {22},
      YEAR = {2015},
    NUMBER = {2},
     PAGES = {Paper 2.38, 44},
}

@incollection {LeGallMiermont12,
    AUTHOR = {Le Gall, J. F. and Miermont, G.},
     TITLE = {Scaling limits of random trees and planar maps},
 BOOKTITLE = {Probability and statistical physics in two and more
              dimensions},
    SERIES = {Clay Math. Proc.},
    VOLUME = {15},
     PAGES = {155--211},
 PUBLISHER = {Amer. Math. Soc., Providence, RI},
      YEAR = {2012},
}

@article {ChassaingDurhuus06,
    AUTHOR = {Chassaing, P. and Durhuus, B.},
     TITLE = {Local limit of labeled trees and expected volume growth in a
              random quadrangulation},
   JOURNAL = {Ann. Probab.},
  FJOURNAL = {The Annals of Probability},
    VOLUME = {34},
      YEAR = {2006},
    NUMBER = {3},
     PAGES = {879--917},
}

@article {ChassaingSchaeffer04,
    AUTHOR = {Chassaing, P. and Schaeffer, G.},
     TITLE = {Random planar lattices and integrated super{B}rownian
              excursion},
   JOURNAL = {Probab. Theory Related Fields},
  FJOURNAL = {Probability Theory and Related Fields},
    VOLUME = {128},
      YEAR = {2004},
    NUMBER = {2},
     PAGES = {161--212},
}

@article {Bernardi07,
    AUTHOR = {Bernardi, O.},
     TITLE = {Bijective counting of tree-rooted maps and shuffles of
              parenthesis systems},
   JOURNAL = {Electron. J. Combin.},
  FJOURNAL = {Electronic Journal of Combinatorics},
    VOLUME = {14},
      YEAR = {2007},
    NUMBER = {1},
     PAGES = {Research Paper 9, 36},
     }

@article {BernardiFusy_Hypermaps,
    AUTHOR = {Bernardi, O. and Fusy, E.},
     TITLE = {Unified bijections for planar hypermaps with general
              cycle-length constraints},
   JOURNAL = {Ann. Inst. Henri Poincar\'e{} D},
  FJOURNAL = {Annales de l'Institut Henri Poincar\'e{} D. Combinatorics,
              Physics and their Interactions},
    VOLUME = {7},
      YEAR = {2020},
    NUMBER = {1},
     PAGES = {75--164},
}

@misc{T25,
      title={Ising model with external magnetic field on random planar maps: critical exponents}, 
      author={Tokka, N.},
      year={2025},
      archivePrefix={arXiv},
      eprint={2511.03688},
      primaryClass={math.PR},
      note = "arXiv:2511.03688",
}

@article{LeGall_Uniqueness,
  title={Uniqueness and universality of the Brownian map},
  author={Le Gall, J. F.},
  journal={Ann. Probab.},
  pages={2880--2960},
  year={2013},
  publisher={JSTOR}
}

@article{Miermont_Brownian,
  title={The Brownian map is the scaling limit of uniform random plane quadrangulations},
  author={Miermont, G.},
  journal={Acta Math.},
  volume={210},
  number={2},
  pages={319--401},
  year={2013}
}


\bigskip

{ {\sc{Marie Albenque:}} \href{mailto:malbenque@irif.fr}{malbenque@irif.fr}}\\
{IRIF, CNRS, Université Paris Cité, F-75013 Paris, France.}

\medskip

{ {\sc{Laurent Ménard}}: \href{mailto:laurent.menard@normalesup.org}{laurent.menard@normalesup.org}}\\
{Modal'X, UPL, Université Paris-Nanterre, F-92000 Nanterre, France.}

\medskip

{ {\sc{Nicolas Tokka:}} \href{mailto:n.tokka@parisnanterre.fr}{n.tokka@parisnanterre.fr}}\\
{Modal'X, UPL, Université Paris-Nanterre, F-92000 Nanterre, France.}

\end{document}